\documentclass[hidelinks,onefignum,onetabnum]{siamart251216}

\newsiamremark{example}{Example}
\newsiamthm{assumption}{Assumption}

\def\norm#1{\|#1\|}

\newcommand{\sipg}{SIPG\xspace}
\newcommand{\swip}{SWIP\xspace}
\newcommand{\sipgL}{SIPG-L\xspace}
\newcommand{\swipL}{SWIP-L\xspace}

\newcommand{\swipDL}{SWIPD-L\xspace}
\newcommand{\swipD}{SWIPD\xspace}
\newcommand{\swipDPL}{SWIPDP-L\xspace}
\newcommand{\swipDP}{SWIPDP\xspace}
\newcommand{\fem}{FEM\xspace}
\newcommand{\femL}{FEM-L\xspace}
\newcommand{\femC}{FEM-C\xspace}

\newcommand{\proj}{\Pi}
\newcommand{\coerexp}{\beta}
\newcommand{\stabfactor}{\alpha}
\newcommand{\limpff}{\widehat{\pf_K}}
\newcommand{\limpf}{\widehat{\pf_h^n}}
\newcommand{\pfmean}[1][]{\overline{\pf^{#1}_K}}

\newcommand{\pfmeanstar}[1][]{\overline{\pf^{#1}_{K^\star}}}
\newcommand{\pfmeanisec}[1][]{\overline{\pf^{#1}_{K_\isec}}}
\newcommand{\pfmeanisecpm}{\overline{\pf_{K_\isec^\pm}}}
\newcommand{\pfmeanisecp}{\overline{\pf_{K_\isec^+}}}
\newcommand{\pfmeanisecm}{\overline{\pf_{K_\isec^-}}}

\newcommand{\grid}{{\mathcal{T}_h}}
\newcommand{\entity}{K}
\newcommand{\elem}{\entity}

\newcommand{\elemin}{\elem^{-}}
\newcommand{\elemout}{\elem^{+}}

\newcommand{\isec}{e}

\newcommand{\RRR}{\mathbb{R}}

\newcommand{\R}{\RRR}

\newcommand{\tol}{\text{TOL}}

\newcommand{\vect}[1]{\boldsymbol{#1}}

\newcommand{\nbold}{\boldsymbol{n}}

\newcommand{\ics}{\vect{x}}

\newcommand{\vjump}[1]{ [ {#1} ] }
\newcommand{\aver}[1]{\{ {#1} \} }

\newcommand{\haver}[1]{\langle {#1} \rangle }

\newcommand{\chem}{\upsilon}

\newcommand{\Peclet}{\mathrm{Pe}}
\newcommand{\Cahn}{\mathrm{Cn}}
\newcommand{\mob}{M}
\newcommand{\polyord}{p}
\newcommand{\pena}{\eta_\isec}
\newcommand{\dt}{\tau}

\newcommand{\faceCount}[1]{N_{#1}}

\newcommand{\Kplus}{{K^+_e}}
\newcommand{\Kminus}{{K^-_e}}
\newcommand{\varphiplus}{\varphi_{|\Kplus}}
\newcommand{\varphiminus}{\varphi_{|\Kminus}}

\newcommand{\pf}[0]{\psi}
\newcommand{\tilpf}[0]{\tilde{\psi}}

\newcommand{\N}[0]{\mathbb{N}}

\newcommand{\mb}[1]{\boldsymbol{#1}}

\newcommand{\ome}[0]{\Omega}
\newcommand{\mbu}[0]{\mb{u}}
\newcommand{\mbx}[0]{\mb{x}}

\newcommand{\vp}[0]{\varphi}

\newcommand{\Po}[0]{\mathbb{P}}

\newcommand{\innerProd}[2]{\left\langle #1, #2\right\rangle}

\usepackage{dune}

\usepackage{type1cm}        
\usepackage{makeidx}         
\usepackage{graphicx}        
\usepackage{multicol}        
\usepackage[bottom]{footmisc}

\usepackage{newtxtext}       %
\usepackage[varvw]{newtxmath}       

\usepackage{todonotes}

\usepackage{lipsum}
\usepackage{amsfonts}
\usepackage{graphicx}
\usepackage{epstopdf}
\usepackage{orcidlink}
\usepackage{algorithmic}
\ifpdf
  \DeclareGraphicsExtensions{.eps,.pdf,.png,.jpg}
\else
  \DeclareGraphicsExtensions{.eps}
\fi

\newsiamremark{remark}{Remark}
\newsiamremark{hypothesis}{Hypothesis}
\crefname{hypothesis}{Hypothesis}{Hypotheses}
\newsiamthm{claim}{Claim}
\newsiamremark{fact}{Fact}
\crefname{fact}{Fact}{Facts}

\headers{DG Scheme for Cahn-Hilliard with Maximum Principle}{J.K. Gunnarsson and R. Kl\"ofkorn}

\title{A Discontinuous Galerkin Scheme for the Cahn-Hilliard Equations with
Discrete Maximum Principle for Arbitrary Polynomial Order
\thanks{Submitted to the editors \today.
\funding{This work was funded by the Swedish Research Council under contract AI-Twin~(2024-04904).}}}

\author{Jimmy Kornelije Gunnarsson\thanks{Centre for Mathematical Sciences, Lund
University, Box 117, 22100 Lund, Sweden (\email{jimmy\_kornelije.gunnarsson@math.lu.se}
\orcidlink{0009-0000-7610-6360}, \email{robertk@math.lu.se} \orcidlink{0000-0001-9664-0333}) }.
\and Robert Kl\"ofkorn\footnotemark[2].}

\usepackage{amsopn}

\makeindex             

\begin{document}

%
%
%
\maketitle

\begin{abstract}
We propose a structure-preserving discontinuous Galerkin scheme for the
Cahn--Hilliard equations with degenerate mobility based on the Symmetric
Weighted Interior Penalty  formulation. By evaluating the mobility at cell
averages rather than as a piecewise polynomial, the proposed scheme preserves
strict degeneracy and yields a coercivity constant that is independent of the
mobility, removing the need for regularization. Moreover, we establish existence
of discrete solutions even with degeneracy via a Leray--Schauder fixed-point
argument, and show that the scheme satisfies a provable discrete maximum
principle at arbitrary polynomial order $\polyord$ when combined with the
Zhang--Shu scaling limiter for $\polyord > 0$ and from the scheme alone for
$\polyord = 0$. Mass conservation and energy dissipation are established for the
unlimited scheme; for the limited variant, we discuss observed energy
dissipation for $\polyord \in \{1,2\}$ and potential theoretical solutions.
Numerical experiments confirm optimal convergence rates of order $\polyord$ in
$H^1$ for $p \in \{1,2\}$, and of order $\polyord+1$ in $L^2$ in a low amplitude
regime with order $2$ for $p=2$ in high amplitude regime and time-dependent low
amplitude regime, and validate  structure-preserving properties with numerical
results.
\end{abstract}



\section{Introduction}
\label{sec:1}
Consider the Lipschitz domain $\Omega \subset \R^d$ for $d \in \{1,2,3\}$ with boundary
$\partial \Omega$ and time domain $[0,T]$ for some $T \in \R^+$; for convenience,
we write $\Omega_T := \Omega \times [0,T]$. The Cahn--Hilliard (CH) equations
for a phase-field $\pf: \Omega_T \to [-1,1]$ read:
\begin{eqnarray}\label{eq:CH}
\partial_t \pf &=& \Peclet^{-1}\nabla \cdot (\mob(\pf) \nabla \chem),
\label{eq:ch1}\\
\chem &=& W'(\pf) - \Cahn^2 \Delta \pf, \label{eq:ch2}
\end{eqnarray}
with initial and boundary conditions:
\begin{eqnarray}
\pf(\cdot,0) &=& \pf^0, \label{eq:ch3} \\
\nabla \pf \cdot \mathbf{n}|_{\partial \Omega} &=& 0, \label{eq:ch5a} \\
\mob(\pf) \nabla \chem \cdot \mathbf{n}|_{\partial \Omega} &=& 0, \label{eq:ch5}
\end{eqnarray}
where $\chem$ is the chemical potential, $\Peclet > 0$ is the P\'{e}clet
number, $\Cahn > 0$ is the Cahn number, $\mob(\pf) = \max\{1-\pf^2,0\}$ is the degenerate
mobility, and $W$ is a double-well potential. There are two popular choices for the
double-well potential, one is the quartic potential~\cite{Acosta:2021,Gunnarsson:20262, Gunnarsson:2026}:
\begin{equation}\label{eq:polynomial}
W(\pf) = \frac{1}{4}(\pf^2-1)^2,
\end{equation}
while the original formulation is the logarithmic potential~\cite{Elliott:2000},
\begin{equation}\label{eq:braggWilliams}
W(\pf) = \frac{\theta}{2}((1+\pf)\ln(1+\pf) + (1-\pf)\ln(1-\pf)) + \frac{\theta_c}{2}(1-\pf^2),
\end{equation}
for parameters $\theta_c \geq \theta > 0$. For the logarithmic potential in
Eq.~\eqref{eq:braggWilliams}, $\pf$ is confined to $(-1,1)$~\cite{Elliott:2000},
whereas for the quartic potential in Eq.~\eqref{eq:polynomial} the solution is
not a priori bounded, since fourth-order partial differential equations do not
admit a strong maximum principle as found in ~\cite{Elliott:2000}
and~\cite[Remark 3.5]{Eikelder:2024} (in contrast to second-order in space
problems such as the heat equation \cite{Thomee:2006} or the Allen--Cahn
equation \cite{Shen:2010}). Consequently, discrete formulations of the CH
equations demand carefully constructed numerical schemes that enforce a Discrete Maximum Principle (DMP) for quartic potentials~\cite{Gunnarsson:2026} in the hope of
satisfying a weak maximum principle as proven in~\cite[Theorem 1]{Elliott:2000},
namely for the quartic potential in Eq.~\eqref{eq:polynomial} alongside the
mobility $M$ for $\pf^0 \in H^1(\Omega)$ with $|\pf^0| \leq 1$ a.e.\ in $\Omega$
then $|\pf(\cdot, t)| \leq 1$ a.e.\ in $\Omega \times [0,T]$. Further analysis
in the weak case has also been carried out in \cite{dai:2016}. Constructing
provable bound-preserving schemes (i.e. having a DMP) is, in general, a highly nontrivial task, as
highlighted in~\cite[Remark 3.5]{Eikelder:2024} and references therein. As
demonstrated in~\cite{Acosta:2021,frank:2018,Gunnarsson:2026} the theorem in
\cite[Theorem 1]{Elliott:2000} does not translate directly to standard
discretization methods such as the Finite Element Method (\fem) and the
Symmetric Interior Penalty Galerkin (\sipg) discretization, both of which
violate a DMP in numerical simulations. The limited
variants \femL and \sipgL introduced in~\cite{Gunnarsson:2026} restore
boundedness in numerical experiments, yet neither is supported by a formal
proof. For a DMP, only recently have rigorous results emerged, either through
upwind flux design~\cite{Acosta:2021} or through limiter-based
approaches~\cite{Frank:2020,Liu:2024} for the quartic case. A scheme which satisfies the DMP has also been developed for the logarithmic potential~\cite{Chen:2019}.
We also note that the analysis in~\cite{Elliott:2000} does not account for
advection; this has since been addressed at the discrete level
in~\cite{Acosta:2021} by means of upwinding and piecewise constant basis
functions, and we discuss this extension in
Remark~\ref{rem:convectiveBoundedness}. \par
A key limitation of, for instance, the \femL, \swipL and \sipgL schemes
in~\cite{Gunnarsson:2026} is that their fluxes across cell intersections do not
readily ensure a DMP for the cell averages. In particular, the Zhang--Shu
scaling limiter employed in~\cite{Gunnarsson:20262, Gunnarsson:2026} (and as a
second step in \cite{Liu:2024}) requires that the cell averages themselves
satisfy a DMP. This motivates the use of the Symmetric Weighted Interior Penalty
(\swip) formulation, which strictly imposes zero flux on intersections where the
cell average attains a global maximum or minimum by exploiting the degeneracy of
the mobility $M$. The \swip discretization was originally discussed in
\cite{Burman:2006} for isotropic diffusion and subsequently generalized to the
anisotropic setting in \cite{Ern:2008}. A key assumption in \cite{Ern:2008} is
to consider a piecewise constant diffusion tensor. Although the \swip
discretization was briefly discussed in \cite{frank:2018} and for the first time
applied to the CH equations in \cite{Gunnarsson:20262, Gunnarsson:2026}, the
prescribed formulations in the latter employ a piecewise polynomial mobility
evaluation that necessitates a regularisation assumption. This requirement can
be circumvented by evaluating the mobility at the cell averages, thereby
permitting strict piecewise degeneracy. As demonstrated in
Theorem~\ref{thm:boundedVhk}, this modification enables a provable DMP at the
cell-average level, and thus, also for the phase-field approximation following a
scaling limiter discussed in Section \ref{sec:scalinglimiter}.
%
In the present work, we propose a novel DG formulation -- the \swip Diffusive
Projection (\swipDP) scheme combined with the Zhang--Shu scaling limiter
(\swipDPL) -- that provably satisfies a DMP for the CH equations at arbitrary
polynomial order in Theorem~\ref{thm:boundedVhk}. We establish existence of
solutions for the unlimited \swipDP scheme and, by virtue of an energy estimate,
for the limited variant as well. This guarantee comes at the expense of not
necessarily ensuring monotone energy dissipation for $\polyord > 0$, although
this property is consistently observed in numerical experiments. The resulting
scheme furnishes a robust framework for simulating the CH equations, including
advective extensions, and can be naturally incorporated into coupled
formulations as in~\cite{Gunnarsson:20262,Gunnarsson:2026}.
\subsection{Outline}
The outline of this paper is as follows. In Sec.~\ref{sec:discretization} we
present the spatial discretization, and operators used throughout the paper. We
continue then with temporal discretization in Sec.~\ref{sec:scheme} to define
the \swipDP scheme by providing a schematic energy dissipation rule, mass
conservation, and also existence of solutions. A DMP is then derived in
Sec.~\ref{sec:dmp} by means of using a scaling limiter for $\polyord \geq 1$ and
without a limiter for $\polyord = 0$, and we discuss our approach also to other
schemes found in the literature. Finally, to validate our theoretical results,
we provide numerical experiments in Sec.~\ref{sec:numerics} which show numerical
evidence for our findings.

\section{Spatial Discretization}\label{sec:discretization} In this section, we
present the spatial discretization and the functional framework for the CH
equation system. Firstly, we introduce the spatial discretization, choice of
basis functions, operators and their properties. The primary result concerning
the operators is to motivate the design of the \swipDPL scheme and to establish
coercivity of the resulting trilinear form arising from discretizing the
mobility term in Thm.~\ref{thm:coercivity}. The overall coercivity proof and the
remaining details follow a similar procedure to that in \cite{Gunnarsson:2026}.
\subsection{Notation} \label{sec:notation} Let the spatial domain $\ome$ be
partitioned into a union of $N$ non-intersecting elements $\elem$ forming
a mesh $\grid = \cup_{i = 1}^{N} \elem_i$. Then we denote by $\Gamma_i$
with unit normal $\mb{n}$ the set of all intersections between two elements of
the grid $\grid$, and the set of all intersections, also with the boundary of
the domain $\Omega$, is denoted by $\Gamma$. For an intersection $\isec \in
\Gamma$ we denote the adjacent elements with $\elemin_\isec$ and
$\elemout_\isec$ ($\elemin_\isec = \elemout_\isec$ for $\isec \in \Gamma
\setminus \Gamma_i$) and define $h_{\isec}$ as the harmonic mean of the
neighboring element areas divided by the face length $|\isec|$:
\begin{equation}\label{eq:harmonicFace}
  h_\isec := \frac{2|\elemout_\isec| |\elemin_\isec|}{|\isec|(|\elemin_\isec| + |\elemout_\isec|)} \qquad \forall \isec \in \Gamma.
\end{equation}

The global mesh width is then defined as $h = \max_{\isec \in \Gamma} h_\isec$.
\begin{assumption}[Mesh regularity]\label{amp:mesh} The mesh $\grid$ is
shape-regular and uniform in the sense we only consider right-angled isosceles
triangles and quadrilaterals for $\Omega \subset \R^2$ and right-angled
isosceles tetrahedra or hexahedra for $\Omega \subset \R^3$ with elements of
equal shape.
\end{assumption}
We establish Assumption.~\ref{amp:mesh} to simplify the algebra in
Thm.~\ref{thm:coercivity} and to use $h_\isec = \frac{|\elem_e|}{|\isec|}$.
Moreover, we only investigate under this mesh assumptions for the numerical
experiments in Sec.~\ref{sec:numerics}, in particular, Exs.~\ref{ex:trig} uses
quadrilaterals while Exs.~\ref{ex:spinodal} and~\ref{ex:merging} uses
right-angled isosceles triangles. However, Thm.~\ref{thm:coercivity} hold for
general meshes with minor modifications, in particular, by adjusting the
estimate following a modification of the trace inverse inequality in
Eq.~\eqref{eq:traceinverseConstant} by accounting for element angles
following~\cite{Epshteyn:2007} or other element-shape properties. The latter,
for instance, may also involve redefining $h_\isec$ to account for it.
Modifications in the estimates also holds in the mesh refinement cases, as
highlighted in~\cite{Gunnarsson:2026}. \par

Following standard \fem notation, we consider a general-order \fem formulation
for the space of trial and test functions:
\begin{equation}
V_h^\polyord = \{ \varphi \in L^2(\grid) : \varphi|_\elem \in \Po^\polyord(\elem),
\forall \elem
\in \grid
\},
\end{equation}
where $\Po^\polyord(\elem)$ denotes a polynomial space of order at most $\polyord$
on the
element $\elem$. Note that this makes $V_h^\polyord$ a broken polynomial space over
the grid $\grid$. We also introduce the inner product $\innerProd{\cdot}{\cdot}_\elem$
for $\elem \in \grid$ and $\innerProd{\cdot}{\cdot}_\grid$ as
\begin{equation}
  \innerProd{\varphi}{\phi}_\elem = \int_\elem \varphi \phi \, d\ics, \qquad \innerProd{\varphi}{\phi}_\grid
= \sum_{\elem \in \grid} \innerProd{\varphi}{\phi}_\elem,
\end{equation}
for scalar functions $\varphi, \phi \in V_h^\polyord$,
inducing the $L^2$-norm $\innerProd{\varphi}{\varphi}_\elem = \norm{ \varphi }_{L^2(\elem)}^2$.
For brevity, we omit the subscript $\grid$ in the inner product when it is clear
from context.
 We introduce the
operators $\vjump{\cdot}_{\isec}, \aver{\cdot}_{\isec}$, and
$\haver{\cdot}_{\isec}$ for
 $\isec \in \Gamma_i$ as
\begin{equation}\label{eq:jumpaverage}
  \begin{split}
    \vjump{\varphi}_{\isec} = \varphiminus  - \varphiplus, \qquad
    \aver{\varphi}_{\isec} = \frac{1}{2}\left(\varphiminus+\varphiplus\right), \qquad
    \haver{\varphi}_{\isec} = \frac{2\, \varphiplus \, \varphiminus}{\varphiplus
+ \varphiminus},
\\
  \end{split}
\end{equation}
for some $\varphi$, where $\vjump{\cdot}_{\isec}$, $\aver{\cdot}_{\isec}$, and
$\haver{\cdot}_{\isec}$ denote the jump, arithmetic average, and harmonic
average across intersection $\isec$, respectively. For brevity, we drop the
subscript $\isec$ for those when the context is clear and adopt the notation
$\varphi^\pm := \varphi_{|\elem_\isec^\pm}$ with $\varphi_\elem$ denoting the
restriction of $\varphi$ on the element $\elem \in \grid$. Throughout this
paper, we will also use that the harmonic average is similar to $\min$ for $a,b
\geq 0$:
\begin{equation}\label{eq:harmonicAverageEstimates1}
  \min\{a,b\} \leq \frac{2ab}{a + b} \leq 2\min\{a, b\},
\end{equation}
such that, in the special case of $a = b = 0$ we define from the estimate in Eq.~\eqref{eq:harmonicAverageEstimates1} that $\underset{a,b \to 0^+}{\lim} \frac{2ab}{a + b} = 0$ to avoid division by zero via a comparison. We also have
\begin{equation}\label{eq:harmonicAverageEstimates2}
\frac{2ab}{a + b} \leq \frac{a + b}{2} \leq \max\{a, b\},
\end{equation}
which will later be utilized since the upper bound of the harmonic average can be formulated from both of Eq.~\eqref{eq:harmonicAverageEstimates1} and Eq.~\eqref{eq:harmonicAverageEstimates2} such that $\frac{2ab}{a + b} \leq 2a$ and $\frac{2ab}{a + b } \leq 2b$. \par

Moreover, we introduce the notation $\varphi_\oplus := \max{\{0,\varphi\}}$ and
$\varphi_\ominus := \min{\{0, \varphi\}}$ to denote the positive and negative
parts of a function, respectively.
Finally, we introduce the following notation for the set of
functions in $L^\infty(\grid) := \{\varphi: \sup_{\ics \in \grid} |\varphi(\ics)| < \infty\}$ that take values in $[-1,1]$ a.e.\ in $\grid$:
\begin{equation}
  L_1^\infty(\grid) := \{ \varphi \in L^\infty(\grid) : \norm{\varphi}_{L^\infty(\grid)}
\leq 1 \},
\end{equation}
and denote set membership in $L_1^\infty(\grid)$ by the following notation:
\begin{equation}
  \norm{\varphi}_{L_1^\infty(\grid)} = \begin{cases}
  \norm{\varphi}_{L^\infty(\grid)} & \text{if } \varphi \in L_1^\infty(\grid), \\
  \infty & \text{otherwise},
  \end{cases}
\end{equation}
and we say that $\varphi$ is bounded in $L_1^\infty(\grid)$
if $\norm{\varphi}_{L_1^\infty(\grid)} < \infty$.
\subsection{Discontinuous Galerkin}
We proceed by constructing a general formulation for the spatial discretization that
is suitable for practical implementation.
\begin{definition}[Cell--averaged projection]\label{def:cellavgproj} Consider
$\varphi \in V_h^\polyord$ and $\elem \in \grid$. We define the cell--averaged
projections $\Pi_\elem^0: \Po^\polyord(\elem) \to \Po^0(\elem)$ and
$\Pi_\grid^0: V_h^\polyord \to V_h^0$ as
\begin{equation}
  \Pi_\elem^0 \varphi = |\elem|^{-1} \int_\elem \varphi \, d\ics, \qquad \Pi_\grid^0
\varphi(\ics) = \sum_{\elem \in \grid} \Pi_\elem^0 \varphi \, \chi_\elem(\mbx),
\end{equation}
where
\begin{equation}
  \chi_\elem(\mbx) := \begin{cases}
  1 & \text{if } \mbx \in \elem, \\
  0 & \text{otherwise}.
  \end{cases}
\end{equation}
\end{definition}
Def.~\ref{def:cellavgproj} introduces an operator $\Pi_\elem^0$ that permits us
to extract the cell-averaged value of a polynomial $\varphi \in
\Po^\polyord(\elem)$ which will be utilized in Sec.~\ref{sec:dmp} to establish
the DMP. Moreover, by means of orthogonality, we can also decompose any
polynomial $\varphi \in \Po^\polyord(\elem)$ into a cell-averaged part
$\overline{\varphi_\elem} := \proj_\elem^0 \varphi $ and a zero-mean part
$\tilde{\varphi}_\elem:= \varphi - \proj_\elem^0 \varphi$ which are orthogonal
to each other. This, in turns, can be exploited for the DG-\fem space
$V_h^\polyord$ for ease of operations with the scaling limiter in
Sec.~\ref{sec:scalinglimiter} without interfering with the cell-averaged part
$\proj_\elem^0 \varphi$.
\begin{lemma}[Orthogonality]\label{lem:orthogonality}
Suppose that we have an orthogonal basis $\{\phi_j\}_{j = 0}^{N_p}$ with $N_{p} := \dim(\Po^p(\elem)) - 1$ with respect to
the $L^2$ inner product for $\Po^\polyord(\elem)$ for each $\elem \in \grid$
such that $\innerProd{\phi_i}{\phi_j}_\elem = \delta_{ij} |\elem|$ and $\phi_0
= 1$. Then for any $\pf \in \Po^\polyord(\elem)$ such that $\pf = \sum_{j = 0}^{N_p}
c_j \phi_j$ we have that $\innerProd{\Pi_\elem^0 \pf}{1}_\elem = c_0$ and $\innerProd{\pf - \Pi_\elem^0
\pf}{\phi_0}_\elem = 0$.
\end{lemma}
\begin{proof}
We prove the first statement. By definition,
\begin{equation}
  \innerProd{\Pi_\elem^0 \pf}{1}_\elem = |\elem|^{-1} \int_\elem \sum_{j = 0}^{N_p} c_j \phi_j \, d\ics
= |\elem|^{-1}  \sum_{j = 0}^{N_p} c_j \innerProd{\phi_j}{\phi_0}_\elem = c_0.
\end{equation}
The second statement follows immediately from the orthogonality of the basis.
\end{proof}
 More generally following Lem.~\ref{lem:orthogonality},
we write $\tilde{\pf} = \sum_{\elem \in \grid} \tilde{\pf}_\elem \chi_\elem$
and note that $\tilde{\pf} \in V_h^\polyord \setminus (V_h^0 \setminus \{0\})$ for $\polyord \geq
0$, with $\tilde{\pf} = 0$ if $\polyord = 0$ as $\proj_\grid^0$ is the identity map over $V_h^0$. We also have the relation $\proj_\grid^0 \pf = \pf - \tilde{\pf}$; recall that $V_h^\polyord$ is a broken polynomial space since continuity across
interior element faces is not imposed.\par

Let $\pf_h \in V_h^\polyord$ denote a fixed discrete phase-field. Following the formulation
in
\cite{Ern:2008}, we introduce a piecewise constant mobility $\mob(\proj_\grid^0 \pf_h)$
that serves as an isotropic piecewise degenerate diffusion coefficient. Note that
$\mob(\pfmean) \in \Po^0(\elem)$
and $\mob(\Pi_\grid^0 \pf_h) \in L^\infty(\grid)$ by definition. In the spirit of
the \swip formulations in~\cite{Gunnarsson:20262, Gunnarsson:2026}, we introduce
the trilinear form $b: L^\infty(\grid) \times V_h^\polyord \times V_h^\polyord \to
\R$:
\begin{eqnarray}
b(\mob(\Pi_\grid^0 \pf_h)&,&\chem_h, \vp) = \int_\grid \mob(\Pi_\grid^0 \pf_h)\,
\nabla \chem_h \cdot \nabla \vp\, d\ics \notag \\
    &+& \sum_{\isec \in \Gamma_i} \int_\isec \haver{\mob(\Pi_\grid^0 \pf_h)}\left(
        \frac{\pena }{h_\isec} \vjump{\chem_h} \vjump{\vp}
    - \aver{\nabla \chem_h \cdot \mb{n}^+} \vjump{\vp} - \aver{\nabla \vp \cdot \mb{n}^+} \vjump{\chem_h}\right)
     ds,
    \label{eq:diffmob}
\end{eqnarray}
where $\pena > 0$ is a penalty parameter over $\isec \in \Gamma_i$
which we establish in Theorem~\ref{thm:coercivity}. Moreover,
the scaling factor
\begin{equation}
  \haver{\mob(\Pi_\grid^0 \pf_h)} := \frac{2 \mob(\Pi_\grid^0 \pf_h)^{+} \mob(\Pi_\grid^0 \pf_h)^{-}}
        {\mob(\Pi_\grid^0 \pf_h)^{+} + \mob(\Pi_\grid^0 \pf_h)^{-}}
\end{equation}
follows from choosing the weights
\begin{equation}
  \label{eq:weights}
w_\isec^\mp := \frac{\mob(\pfmeanisecpm)}{\mob(\pfmeanisecp)
+ \mob(\pfmeanisecm)}
\end{equation}
introduced as \swip
in \cite{Ern:2008} for treating
a diffusion-type formulation
for DG, whereas a standard \sipg scheme uses $w_e^\pm = \frac{1}{2}$.
This introduction, for the CH case, is further discussed in~\cite{Gunnarsson:2026}.\par
Coercivity of $\tilde{b}(\cdot, \cdot) = b(\mob(\Pi_\grid^0 \pf_h), \cdot, \cdot)$ for a fixed
$\pf_h$ can only be established over the set
\begin{equation}\label{eq:elliptic}
  \tilde{\grid}(\pf_h) := \grid \setminus \{\elem: \mob(\pfmean) = 0\} = \{\elem \in \grid:
\mob(\pfmean) > 0\},
\end{equation}
and, without loss of generality, we assume that $\tilde{\grid}(\pf_h) \neq
\emptyset$. This is a natural condition required to obtain meaningful bounds for
the corresponding constants. We later also establish Eq.~\eqref{eq:elliptic} in
a solution-dependent setting in Assumption~\ref{amp:coercivityDomain} based on
numerical observations of solutions.
\begin{theorem}[Coercivity] \label{thm:coercivity} Let $\pf_h \in V_h^\polyord$
be a fixed discrete phase-field and suppose that there exists at least one
intersection $e$ with $|\pfmeanisecpm| < 1$, such that $\tilde{\grid}(\pf_h)
\neq \emptyset$. Then $\tilde{b}(\cdot, \cdot) := b(\mob(\Pi_\grid^0 \pf_h),
\cdot, \cdot)$ defined in Eq.~\eqref{eq:diffmob} is coercive over the space
\begin{equation}
  \tilde{V}_h^\polyord(\pf_h) = \{\varphi \in L^2(\tilde{\grid}(\pf_h)): \varphi|_\elem \in \Po^\polyord(\elem),
\forall \elem \in \tilde{\grid}(\pf_h) \} \setminus \{1\},
\end{equation}
if the penalty parameter $\eta_\isec$ satisfies:
\begin{equation}
  \eta_\isec = \eta \geq 2 \faceCount{\partial} C_T \qquad \forall \isec \in \Gamma_i,
\end{equation}
where $C_T$ is the trace inverse constant defined in
Eq.~\eqref{eq:traceinverseConstant} for the element type in $\grid$ and
$\faceCount{\partial}$ maximum number of interior faces per element, in
particular $C_T$ and $\faceCount{\partial}$ are both known and constant
following Assumption~\ref{amp:mesh} and so also $\eta$.
\end{theorem}

\begin{proof}
We consider the case with $p > 0$, and reserve case $p = 0$ for Remark.~\ref{rem:coercivityP0}. We introduce the mobility-weighted DG semi-norm:
\begin{equation}\label{eq:dgnorm}
  \norm{\chem_h }^2_{DG} := \norm{ \sqrt{\mob(\Pi_\grid^0 \pf_h)} \nabla \chem_h
}_{L^2(\grid)}^2
+ \sum_{\isec \in \Gamma_i} \int_\isec \frac{\haver{\mob(\Pi_\grid^0 \pf_h)}}{h_\isec}
\vjump{\chem_h}^2
ds,
\end{equation}
and consider the inequality:
\begin{equation}
  \tilde{b}(\chem_h, \chem_h) \geq \tilde{C} \norm{ \chem_h }^2_{DG}, \quad \forall
\chem_h \in \tilde{V}_h^\polyord(\pf_h),
\end{equation}
for some $\tilde{C} > 0$. For the trilinear form Eq.~\eqref{eq:diffmob} we have:
\begin{align*}
\tilde{b}(\chem_h, \chem_h) &= \norm{\sqrt{\mob(\Pi_\grid^0 \pf_h)} \nabla \chem_h}_{L^2(\grid)}^2
+ \sum_{\isec \in \Gamma_i} \int_\isec \frac{\pena \haver{\mob(\Pi_\grid^0 \pf_h)}}{h_\isec}\vjump{\chem_h}^2
ds \\
&\quad - \sum_{\isec \in \Gamma_i} \int_\isec 2\haver{\mob(\Pi_\grid^0 \pf_h)} \aver{\nabla
\chem_h
\cdot \mathbf{n}^+}\vjump{\chem_h}ds.
\end{align*}
First, we bound the cross term $\int_\isec \haver{\mob(\Pi_\grid^0 \pf_h)} \aver{\nabla
\chem_h \cdot
\mathbf{n}^+} \vjump{\chem_h} ds$ in absolute value. Applying the triangle inequality,
we obtain:
\begin{equation} \label{eq:cauchySchwarz}
  \left|\sum_{\isec \in \Gamma_i}\int_\isec \haver{\mob(\Pi_\grid^0 \pf_h)} \aver{\nabla
\chem_h
\cdot \mathbf{n}^+}
\vjump{\chem_h} ds\right| \leq \sum_{\isec \in \Gamma_i} \left|\int_\isec \haver{\mob(\Pi_\grid^0
\pf_h)}
\aver{\nabla \chem_h \cdot \mathbf{n}^+} \vjump{\chem_h} ds\right|,
\end{equation}
and an application of the Cauchy-Schwarz inequality gives for $\coerexp \in [0,1]$:
\begin{eqnarray*}
  &|&\int_\isec \haver{\mob(\Pi_\grid^0 \pf_h)} \aver{\nabla \chem_h \cdot \mathbf{n}^+} \vjump{\chem_h}
ds| \\ &\leq& \norm{ \haver{\mob(\Pi_\grid^0 \pf_h)}^{1-\coerexp} \aver{\nabla \chem_h
\cdot \mathbf{n}^+}}_{L^2(\isec)}
\norm{\haver{\mob(\Pi_\grid^0 \pf_h)}^\coerexp \vjump{\chem_h}}_{L^2(\isec)},
  \nonumber
\end{eqnarray*}
for every $\isec \in \Gamma_i$. Then, by Young's
 inequality, we obtain the intersection-wise estimate:
\begin{align*}
  &2\norm{ \haver{\mob(\Pi_\grid^0 \pf_h)}^{1-\coerexp} \aver{\nabla \chem_h \cdot \mathbf{n}^+}}_{L^2(\isec)}
\norm{\haver{\mob(\Pi_\grid^0 \pf_h)}^\coerexp \vjump{\chem_h}}_{L^2(\isec)} \\
&\leq h_\isec\epsilon_\isec \norm{\haver{\mob(\Pi_\grid^0 \pf_h)}^{1- \coerexp} \aver{\nabla
\chem_h \cdot
\mathbf{n}^+}}_{L^2(\isec)}^2 + \frac{1}{\epsilon_\isec h_\isec} \norm{\haver{\mob(\Pi_\grid^0
\pf_h)}^\coerexp
\vjump{\chem_h}}_{L^2(\isec)}^2,
\end{align*}
for an arbitrary $\epsilon_\isec > 0$ on each intersection
 $\isec \in \Gamma_i$. Then we seperate out the mobility
\begin{equation}\label{eq:mobilitySeperation}
  \norm{\haver{\mob(\Pi_\grid^0 \pf_h)}^{1-\coerexp} \aver{\nabla \chem_h \cdot \mathbf{n}^+}}_{L^2(\isec)}^2
= \haver{\mob(\Pi_\grid^0 \pf_h)}^{2(1-\coerexp)}\norm{\aver{\nabla \chem_h
\cdot \mathbf{n}^+}}_{L^2(\isec)}^2,
\end{equation}
and expanding the average gives:
\begin{equation}
  \norm{\aver{\nabla \chem_h \cdot \mathbf{n}^+}}_{L^2(\isec)}^2  \leq \frac{1}{2} \left(
\norm{\nabla \chem_h^+ \cdot \mathbf{n}^+}_{L^2(\isec)}^2 + \norm{\nabla \chem_h^- \cdot
\mathbf{n}^+}_{L^2(\isec)}^2 \right).
\end{equation}
The trace inverse inequality with degree $p-1$ from~\cite{Ainsworth:2009,Riviere:2008,Warburton:2003} and~\cite[Tab 3.1]{Hillewaert:2013} gives:
\begin{equation}\label{eq:traceinverseConstant}
  C_T := \begin{cases}
  \frac{p(p+d-1)}{d} & \text{if $d$-simplex mesh} , \\
  p^2 & \text{if quadrilateral or hexahedral mesh},
  \end{cases}
\end{equation}
more general definitions are also found in~\cite[Tab 3.1]{Hillewaert:2013}
depending on the mesh geometry. For the trace inverse inequality estimate, we
get due to Eq.~\eqref{eq:harmonicFace} under Assumption~\ref{amp:mesh}:
\begin{equation}
  \frac{1}{2} \left(
\norm{\nabla \chem_h^+ \cdot \mathbf{n}^+}_{L^2(\isec)}^2 + \norm{\nabla \chem_h^- \cdot
\mathbf{n}^+}_{L^2(\isec)}^2 \right) \leq \frac{C_T}{2h_\isec}
(\norm{\nabla \chem_h}_{L^2(\elem_\isec^+)}^2 + \norm{\nabla \chem_h}_{L^2(\elem_\isec^-)}^2).
\end{equation}
To recover the weighted volumetric term, we observe that since $\mob(\pfmean) > 0$ for $\elem \in \tilde{\grid}(\pf_h)$ we have that
\begin{equation}\label{eq:mobilityCombination}
  \norm{\nabla \chem_h}_{L^2(\elem)}^2 = \frac{\norm{\sqrt{\mob(\pfmean)} \nabla
\chem_h}_{L^2(\elem)}^2}{\mob(\overline{\pf_\elem})}.
\end{equation}
We note that $\haver{M(\pfmeanisec)}_\isec \leq 2M(\pfmeanisecpm)$ for $\elem \in \tilde{\grid}(\pf_h)$ following the harmonic average estimate in Eqs.~\eqref{eq:harmonicAverageEstimates1} and~\eqref{eq:harmonicAverageEstimates2},
which yields
\begin{equation}\label{eq:mobilityLimit}
\frac{\haver{M(\overline{\pf_\elem})}_\isec^{2(1-\coerexp)}}{M(\overline{\pf_{\elem_\isec^\pm}})}
\leq 2^{2(1 - \coerexp)} M(\pfmeanisecpm)^{1-2\coerexp} \leq 2^{2(1 - \coerexp)} \max_{\elem \in \grid} M(\pfmean)^{1-2\coerexp},
\end{equation}
and for Eq.\eqref{eq:mobilityLimit} to be bounded from above we require
$\coerexp \in [0,\frac{1}{2}]$, since for $\beta \in (\frac{1}{2}, 1]$ we get
$\underset{|\pfmean| \to 1}{\lim} M(\pfmean)^{1 - 2\beta} \to \infty$, which
also removes mobility-dependency over $\tilde{\grid}$ as
$M(\pfmean)^{1-2\coerexp} \leq 1$. Combining Eqs.~\eqref{eq:mobilitySeperation},
\eqref{eq:mobilityCombination} and \eqref{eq:mobilityLimit} gives a factor of
$2^{2(1 - \coerexp)}$ to recover the volumetric term $\norm{\sqrt{\mob(\pfmean)}
\nabla \chem_h}_{L^2(\elem)}^2$ in the estimate of the cross term. Next, since,
$\mob(\pfmean) > 0$ for $\elem \in \tilde{\grid}(\pf_h)$ and the
contributions from $\grid \setminus \tilde{\grid}(\pf_h)$ vanish, we sum over
$\elem \in \grid$ without loss of generality. For the mobility-weighted DG
semi-norm in Eq.~\eqref{eq:dgnorm}, we obtain the following volumetric term
estimate:
\begin{equation}\label{eq:volumeineq}
\norm{\sqrt{\mob(\pfmean)} \nabla \chem_h}_{L^2(\elem)}^2\left(1 -  \frac{2^{2(1-\coerexp)} C_T}{2}\sum_{\isec
\in
\partial \elem \cap\Gamma_i}\epsilon_\isec \right)
\geq 0 \qquad \forall \elem \in \grid.
\end{equation}
and we note that the face count $\faceCount{\partial \elem} := \sum_{\isec \in
\partial \elem \cap \Gamma_i} 1$ is bounded by $\faceCount{\partial} =
\max_{\elem \in \grid} \faceCount{\partial \elem}$. The inequality in
Eq.~\eqref{eq:volumeineq} is trivially satisfied for $\elem \in \grid \setminus
\tilde{\grid}(\pf_h)$. For $\elem \in \tilde{\grid}(\pf_h)$, it is strictly
positive provided $\epsilon_\isec \leq \frac{4^\beta}{2\faceCount{\partial} C_T
}$, and we pick $\epsilon_\isec = \frac{4^\beta}{4 \faceCount{\partial} C_T }$.
We will use this $\epsilon_\isec$ for the jump term, we obtain the
following:
\begin{equation}\label{eq:penaltyineq}
\int_\isec \frac{\haver{\mob(\Pi_\grid^0 \pf_h)}}{h_\isec}\left( \pena - \frac{\haver{\mob(\Pi_\grid^0
\pf_h)}^{2\coerexp - 1}}{\epsilon_\isec} \right) \vjump{\chem_h}^2 ds \geq 0,
\end{equation}
where we see that $\coerexp = \frac{1}{2}$ removes mobility-dependency. Since there exists an intersection
$e$ with $|\pfmeanisecpm| \neq 1$ by assumption, we have $\haver{\mob(\Pi_\grid^0
\pf_h)} > 0$ for at least one intersection $\isec$, and we choose
\begin{equation}\label{eq:globalPenaltyParam}
  \pena = \eta = 2 \faceCount{\partial} C_T  + \frac{1}{2},
\end{equation}
such that
\begin{equation}\label{eq:penaltyineq2}
  \pena \geq \frac{1}{\epsilon_\isec}
= 2 \faceCount{\partial} C_T ,
\end{equation}
in agreement with standard estimates for the penalty parameter (see \cite{Ainsworth:2009,Riviere:2008,Epshteyn:2007}).
Then there exist two positive constants $C_1, C_2 > 0$ such that
\begin{equation}\label{eq:DGbounds}
  b(\mob(\Pi_\grid^0 \pf_h), \chem_h, \chem_h) \geq C_1 \norm{\sqrt{\mob(\Pi_\grid^0 \pf_h)} \nabla
\chem_h}_{L^2(\grid)}^2
+ C_2 \sum_{\isec \in \Gamma_i} \int_\isec \frac{\haver{\mob(\Pi_\grid^0 \pf_h)}}{h_\isec}
\vjump{\chem_h}^2
ds,
\end{equation}
with $C_1 = C_2 = \frac{1}{2} = \tilde{C}$  we arrive at
\begin{equation}\label{eq:finalDG}
  b(\mob(\Pi_\grid^0 \pf_h), \chem_h, \chem_h) \geq \tilde{C} \norm{\chem_h}^2_{DG}, \quad \forall
\chem_h \in \tilde{V}_h^\polyord(\pf_h),
\end{equation}
which completes the proof.
\end{proof}
The significance of Theorem \ref{thm:coercivity} becomes evident in the derivation
of energy dissipation for the \swipDP scheme in Theorem \ref{thm:energyDPL}, which
requires that $b(\mob(\Pi_\grid^0\pf_h), \chem_h, \chem_h)$ be semi-positive
for $\chem_h, \pf_h \in V_h^\polyord$. Moreover, we also note that Eq.~\eqref{eq:mobilityLimit} is not sharp. However, it is unconditional, despite the constrained domain in
Eq.\eqref{eq:finalDG} since, due to loss of coercivity over $\grid \setminus \tilde{\grid}(\pf_h)$,
the inequality holds by taking $\eta = \eta_\isec$, as in Eq.~\eqref{eq:globalPenaltyParam} following Eq.~\eqref{eq:penaltyineq2}.  In particular, since the constants in Eq.~\eqref{eq:DGbounds}
can be precomputed independently of the data by virtue of Theorem~\ref{thm:coercivity},
the resulting estimate is robust owing to the weights. Moreover, as discussed in~\cite[Remark  3.4]{Gunnarsson:2026}, this formulation avoids the need to estimate
unknown ratios and data-dependent ambiguities arising in~\cite{Gunnarsson:20262, Gunnarsson:2026}
in the choice of the penalty parameter $\pena$. However, this unconditionality is paid for by the factor $2$ in Eq.~\eqref{eq:mobilityLimit} with $\beta = \frac{1}{2}$.
\begin{remark}[Semi-positivity of $\tilde{b}(\cdot,\cdot)$ over $V_h^0$]\label{rem:coercivityP0}
For the lowest polynomial order case, $\polyord = 0$, we have $\nabla \chem_h = \mb{0}$ on each element
$\elem \in \grid$. Consequently, the gradient terms vanish and the mobility-weighted DG semi-norm in
Eq.~\eqref{eq:finalDG} reduces to
\begin{equation}
  \tilde{b}(\chem_h, \chem_h) = \|\chem_h\|_{DG}^2 = \sum_{\isec \in \Gamma_i} \int_\isec
\frac{\haver{\mob(\Pi_\grid^0 \pf_h)}}{h_\isec} \vjump{\chem_h}^2 \, ds, \quad \forall
\chem_h \in V_h^0,
\end{equation}
which vanishes for any globally constant $C_R \in V_h^0 \cap \R$ and locally whenever
$M(\proj_{\grid}^0 \pf_h) = 0$.
\end{remark}
Analogously to Eq.~\eqref{eq:diffmob}, we introduce the
bilinear form $a: V_h^\polyord \times V_h^\polyord \to \mathbb{R}$ for the discretization
of the Laplacian operator in Eq.~\eqref{eq:ch2}:
\begin{eqnarray}
a(\pf_h, \vp) &=
    \int_\grid \nabla \pf_h \cdot \nabla \vp\, d\mb{x}
    + \sum_{\isec \in \Gamma} \int_\isec
        \frac{\pena^\Delta}{h_\isec} \vjump{\pf_h} \vjump{\vp} \notag \\
        &- \aver{\nabla \pf_h \cdot \nbold^+} \vjump{\vp}
        -  \aver{\nabla \vp \cdot \nbold^+} \vjump{\pf_h}
     ds,
    \label{eq:laplace}
\end{eqnarray}
where $\pena^\Delta$ is the local penalty parameter for the Laplacian operator
to assure coercivity (following any standard \sipg formulation in for instance
\cite{Ainsworth:2009,Riviere:2008}). Note that $a(\cdot, \cdot) = b(1, \cdot,
\cdot)$, which reduces directly to a \sipg formulation; consequently, coercivity
of $a$ follows by an analogous argument to that for $b$, with
$\eta_\isec^\Delta$ large enough, for instance $\eta_\isec^\Delta = N_\partial
C_T + \frac{1}{2}$, we have for $C^\Delta > 0$ the inequality:
\begin{equation}\label{eq:coercivity_laplace}
  a(\pf_h, \pf_h) \geq C^\Delta \norm{\pf_h}_{DG, \Delta}^2 = C^\Delta \left( \norm{\nabla \pf_h}_{L^2(\grid)}^2 + \sum_{\isec \in \Gamma} \frac{\pena^\Delta}{h_\isec} \norm{\vjump{\pf_h}}_{L^2(\isec)}^2 \right),
\end{equation}
The superscript
$\Delta$ and DG semi-norm $\norm{\cdot}_{DG, \Delta}$ distinguishes the case for the Laplacian from that of the mobility term in Eq.~\eqref{eq:diffmob}. The penalty parameters $\eta^\Delta_\isec$ and $\eta_\isec$
need not be equal, as discussed in \cite{Gunnarsson:20262, Gunnarsson:2026}. However, we will proceed with $\pena^\Delta = \pena = \eta$ with $\eta = \max\{1,  2 \faceCount{\partial} C_T  + \frac{1}{2} \}$ for simplicity, with $C_T$ as in Eq.~\eqref{eq:traceinverseConstant} for the element type of $\grid$.

\section{Scheme, Existence and Energy Laws}\label{sec:scheme} We now introduce the temporal discretization
of our proposed scheme and that it leads to mass conservation in
Lem.~\ref{lem:mass} and energy dissipation in Lem.~\ref{thm:energyDPL} being
satisfied in the fully discrete level. Moreover, we also provide existence of
solutions in Thm.~\ref{thm:existence}. To obtain these, we split the time
interval $[0, T]$ by taking $N_T$ time-points with increment $\dt :=
\frac{T}{N_T}$. For brevity we write $t^n = n\dt$, for $n = 0, \ldots, N_T$, and
$\pf(t^n, \cdot) = \pf^n(\cdot)$ for time-dependent variables. We perform the
first order in time approximation at time $t^{n+1}$ with $\partial_t \pf_h^{n+1}
= \frac{\pf_h^{n+1} - \pf_h^n}{\dt} + \mathcal{O}(\tau^2)$. \par

Before proceeding, we introduce a decomposition of $W$ to be used in Thm.~\ref{thm:energyDPL}
\begin{equation}\label{eq:convexConcave}
  W_c(\pf) = \frac{1}{4}\pf^4, \quad W_e(\pf) = \frac{1 - 2\pf^2}{4},
\end{equation}
which splits $W$ into a convex and a concave part, denoted as $W_c$ and $W_e$, respectively. The evaluation of
$W_c'$ with $\pf_h^{n+1}$ and $W_e'$ with $\pf_h^{n}$ is referred to as the Eyre splitting~\cite{Eyre:1997,Tierra:2014}. The
\swipDP scheme is to find $(\pf_h^{n+1}, \chem_h^{n+1}) \in V_h^\polyord \times
V_h^\polyord$:
\begin{eqnarray}
  \frac{1}{\dt}\innerProd{\pf_h^{n+1}- \pf_h^{n}}{\vp} + \Peclet^{-1} b(\mob(\Pi_\grid^0
\pf_h^{n+1}), \chem_h^{n+1},
\vp) &=& 0,  \label{eq:ch1ndtDP} \\
  \innerProd{\chem_h^{n+1}}{\xi} - \innerProd{W_c'(\pf_h^{n+1}) + W_e'(\pf_h^{n})}{\xi}
- \Cahn^2 a (\pf_h^{n+1}, \xi)  &=& 0,
  \label{eq:ch2ndtDP}
\end{eqnarray}
for all $\xi, \vp \in V_h^\polyord$ as solutions at time $t = (n+1)\dt$. In the special case of $V_h^0$,
the \swipD scheme introduced in \cite{Gunnarsson:20262} and \swipDP coincide, since
$\Pi_\grid^0$ is the identity map
on $V_h^0$.

In what follows, we establish several fundamental properties of the \swipDP scheme
and demonstrate that any solution satisfies the corresponding physical constraints.
In particular, the energy dissipation established in Corollary \ref{cor:energyEstimate}
serves as an a~priori estimate for the existence result.
\begin{definition}[Phase-field mass]\label{def:mass}
We refer to the quantity
\begin{equation}
  m_{\pf_h} = \frac{1}{|\grid|} \int_\grid \pf_h \, d\mb{x}
\end{equation}
as the phase-field mass.
\end{definition}
\begin{lemma}[Mass conservation]\label{lem:mass}
Suppose that there exists a solution $(\pf_h^{n+1}, \chem_h^{n+1}) \in V_h^\polyord \times
V_h^\polyord$ to the \swipDP scheme in Eqs.~\eqref{eq:ch1ndtDP}--\eqref{eq:ch2ndtDP}.
 Then the phase-field mass is conserved in time, i.e. $m_{\pf_h^{n+1}}
= m_{\pf_h^{n}}$ for all $n \geq 0$.
\end{lemma}
\begin{proof}
Testing Eq.~\eqref{eq:ch1ndtDP} with $\vp = 1$ yields the desired result.
\end{proof}
\begin{assumption}[Coercivity domain]\label{amp:coercivityDomain} A
solution-dependent coercivity domain $\tilde{\grid}[\pf_h^{n}]$ of the bilinear
form $\tilde{b}(\cdot, \cdot) = b(\mob(\Pi_\grid^0 \pf_h^{n}), \cdot, \cdot)$
from Thm.~\ref{thm:coercivity} is non-empty, i.e. there exists at least one
intersection $e$ with $|\overline{\pf_{\elem_\isec^\pm}}^n| < 1$ for any $n$.
\end{assumption}

Assumption.~\ref{amp:coercivityDomain} is a generalization for the coercivity
condition from Thm.~\ref{thm:coercivity}. The premise is based on the
observation that the phase-field $\pf_h$ forms a non-sharp interface profile for
$\Cahn > 0$ given some sufficiently regular initial data $\pf_h^0$. This
property is preserved in numerical observations, as can be seen in for instance Fig.~\ref{fig:rand14}.


\subsection{Energy formulation and a priori bounds}\label{sec:apriori} In this subsection,
we derive a~priori bounds exploiting the energy dissipation property of the \swipDP
scheme. We begin by introducing the following discrete energy functional.
\begin{definition}[Discrete energy]\label{def:energy}
\begin{equation}
  \mathcal{E}[\pf_h] = \frac{\Cahn}{2}
a (\pf_h,
\pf_h)  + \frac{1}{\Cahn}\int_\grid W(\pf_h) d\ics \,
\end{equation}
as the discrete energy.
\end{definition}
With Definitions~\ref{def:cellavgproj} and \ref{def:energy} in hand, we establish
the following discrete energy dissipation result for the \swipDP scheme.
\begin{theorem}[Discrete energy dissipation for \swipDP]\label{thm:energyDPL}
Suppose that there exists a solution $(\pf_h^1, \chem_h^1) \in V_h^\polyord \times
V_h^\polyord$ to the \swipDP scheme
in Eqs.~\eqref{eq:ch1ndtDP}--\eqref{eq:ch2ndtDP} at time $t = \dt$ and that $a$ is coercive and
$\tilde{b}(\vp, \vp) \geq 0$ for all $\vp \in V_h^\polyord$. Then
\begin{equation}
  \mathcal{E}[\pf_h^1] \leq \mathcal{E}[\pf_h^0].
\end{equation}
\end{theorem}
\begin{proof}
Testing Eq.~\eqref{eq:ch1ndtDP} with $\vp = \chem_h^1$ and
Eq.~\eqref{eq:ch2ndtDP} with $\xi = \frac{1}{\dt}(\pf_h^1 - \pf_h^{0})$, we obtain
\begin{eqnarray}
  \frac{1}{\dt}\innerProd{\pf_h^1 - \pf_h^{0}}{\chem_h^1} + \frac{1}{\Peclet}\,
b(\mob(\proj_\grid^0\pf_h^1), \chem_h^1, \chem_h^1) &=& 0, \label{eq:EnergyDiss1} \\
  \frac{1}{\dt}\innerProd{\chem_h^1}{\pf_h^1 - \pf_h^{0}} - \frac{1}{\dt}\innerProd{W_c'(\pf_h^1)
 W_e'(\pf_h^{0})}{\pf_h^1 - \pf_h^{0}}
   - \frac{\Cahn^2}{\dt}\, a(\pf_h^1, \pf_h^1 - \pf_h^{0}) &=& 0. \label{eq:EnergyDiss2}
\end{eqnarray}
Subtracting Eq.~\eqref{eq:EnergyDiss2} from Eq.~\eqref{eq:EnergyDiss1} and multiplying with $\frac{\dt}{\Cahn}$ gives:
\begin{eqnarray}\label{eq:energyDPL}
\notag \\
\frac{1}{\Cahn}\innerProd{W_c'(\pf_h^1) + W_e'(\pf_h^{0})}{\pf_h^1 - \pf_h^{0}} &+& \Cahn a(\pf_h^1, \pf_h^1 - \pf_h^{0})
\notag \\
  &=& -\frac{\dt}{\Cahn \Peclet}\, b(\mob(\proj_\grid^0\pf_h^1), \chem_h^1, \chem_h^1) \leq 0,
\end{eqnarray}
where the inequality follows from $b(\mob(\proj_\grid^0\pf_h^1), \chem_h^1,
\chem_h^1) \geq 0$. We begin by estimating the first term in
Eq.~\eqref{eq:energyDPL}. Following
Eq.~\eqref{eq:convexConcave} we have the convexity and concavity inequalities:
\begin{equation}
  W_c(\pf_h^1) - W_c(\pf_h^{0}) \leq W_c'(\pf_h^1)(\pf_h^1 - \pf_h^{0}), \quad
  W_e(\pf_h^1) - W_e(\pf_h^{0}) \leq W_e'(\pf_h^{0})(\pf_h^1 - \pf_h^{0}),
\end{equation}
thus
\begin{equation}
  \innerProd{W(\pf_h^1) - W(\pf_h^{0})}{1}
  \leq
   \innerProd{W_c'(\pf_h^1) + W_e'(\pf_h^{0})}{\pf_h^1 - \pf_h^{0}}.
\end{equation}
Substituting into Eq.~\eqref{eq:energyDPL} gives:
\begin{equation}\label{eq:final}
  \frac{1}{\Cahn}\innerProd{W(\pf_h^1) - W(\pf_h^{0})}{1} + \Cahn\,
a(\pf_h^1, \pf_h^1 - \pf_h^{0}) \leq 0.
\end{equation}
Then, using
the polarization identity:
\begin{eqnarray}
  a(\pf_h^1, \pf_h^1 &-& \pf_h^{0}) = \frac{1}{2}\bigl(a(\pf_h^1, \pf_h^1)
- a(\pf_h^{0}, \pf_h^{0}) \notag \\  &+& a(\pf_h^1 - \pf_h^{0},
\pf_h^1 - \pf_h^{0})\bigr)
  \geq\frac{1}{2}\bigl(a(\pf_h^1, \pf_h^1) - a(\pf_h^{0}, \pf_h^{0})\bigr),
\end{eqnarray}
which, combined with the previous inequality and Definition~\ref{def:energy} completes
the proof.
\end{proof}
\begin{corollary}[Energy estimate]\label{cor:energyEstimate}
Under the assumptions of Theorem~\ref{thm:energyDPL}, a solution $(\pf_h^1, \chem_h^1)$ to Eqs.~\eqref{eq:ch1ndtDP}--\eqref{eq:ch2ndtDP}
satisfies
\begin{equation}\label{eq:energyEstimate}
  \mathcal{E}[\pf_h^1] + \frac{\dt}{\Peclet \Cahn} b(M(\proj_\grid^0 \pf_h^1), \chem_h^1, \chem_h^1)\leq
\mathcal{E}[\pf_h^0].
\end{equation}
Moreover, $ \tilde{C}\norm{\chem_h^1}^2_{DG} \leq \frac{\Cahn \Peclet}{\dt}\, \mathcal{E}[\pf_h^0]$ and $C^\Delta ||\nabla \pf_h^1||^2_{L^2(\grid)} \leq \frac{1}{\Cahn} \mathcal{E}[\pf_h^1]$.
\end{corollary}
\begin{proof}
From Eq.~\eqref{eq:energyDPL} multiply with $\Cahn$, we obtain:
\begin{equation}
  \innerProd{W_c'(\pf_h^1) + W_e'(\pf_h^{0})}{\pf_h^1 - \pf_h^{0}}
+ \Cahn^2 a(\pf_h^1, \pf_h^1 - \pf_h^{0}) + \frac{\dt}{ \Peclet}\, b(\mob(\proj_\grid^0
\pf_h^1), \chem_h^1, \chem_h^1) = 0.
\end{equation}
Applying  the Eyre
splitting inequality, and the polarization identity as in the proof of Theorem~\ref{thm:energyDPL},
we obtain Eq.~\eqref{eq:energyEstimate}. The bound on $\norm{\chem_h^1}^2_{DG}$
follows from $b(\mob(\proj_\grid^0 \pf_h^1), \chem_h^1, \chem_h^1) \geq \tilde{C}\,
\norm{\chem_h^1}^2_{DG}$ together with $\mathcal{E}[\pf_h^1] \geq 0$. A similar coercivity argument of $a$ following Eq.~\eqref{eq:coercivity_laplace} and since $W(\pf_h^1) \geq 0$ gives the bound $C^\Delta ||\nabla \pf_h^1||^2_{L^2(\grid)} \leq \frac{1}{\Cahn} \mathcal{E}[\pf_h^1]$.
\end{proof}
\begin{lemma}[Uniform $L^2$-bound for $\pf_h^1$]\label{lem:uniformBound}
Under the assumptions of Theorem~\ref{thm:energyDPL}, if $\pf_h^1$ is a solution
to Eq.~\eqref{eq:ch1ndtDP}, then
\begin{equation}\label{eq:uniformBoundCorollary}
  \norm{\pf_h^1}_{L^2(\grid)}^2 \leq 2\Cahn\, \mathcal{E}[\pf_h^0] + \frac{3|\grid|}{2}
=: C_\pf(\pf_h^0, \Cahn, |\grid|),
\end{equation}
where $C_\pf(\pf_h^0, \Cahn, |\grid|)$ is a constant depending only on the initial
phase-field $\pf_h^0$, the Cahn number $\Cahn$, and the measure of the domain $|\grid|$.
\end{lemma}
\begin{proof}
By Corollary~\ref{cor:energyEstimate}, $\mathcal{E}[\pf_h^1] \leq \mathcal{E}[\pf_h^0]$.
Since $a(\pf_h, \pf_h) \geq 0$:
\begin{equation}
  \frac{1}{\Cahn}\int_\grid W(\pf_h^1)\, d\mb{x} \leq \mathcal{E}[\pf_h^1] \leq \mathcal{E}[\pf_h^0].
\end{equation}
For the quartic double-well potential $W(\pf) = \frac{1}{4}(\pf^2 - 1)^2 = \frac{1}{4}\pf^4
- \frac{1}{2}\pf^2 + \frac{1}{4}$, Young's inequality gives $\frac{1}{4}\pf^4 \geq
\pf^2 - 1$, so that $W(\pf) \geq \frac{1}{2}\pf^2 - \frac{3}{4}$. Therefore,
\begin{equation}\label{eq:uniformBoundIntermediate}
  \frac{1}{\Cahn}\left(\frac{1}{2}\norm{\pf_h^1}_{L^2(\grid)}^2 - \frac{3|\grid|}{4}\right)
\leq \frac{1}{\Cahn}\int_\grid W(\pf_h^1)\, d\mb{x}\leq \mathcal{E}[\pf_h^1] \leq \mathcal{E}[\pf_h^0],
\end{equation}
which rearranges to Eq.~\eqref{eq:uniformBoundCorollary}.
\end{proof}


\subsection{Existence}\label{sec:Existence} In this section we use the results from
Section~\ref{sec:apriori} to establish the existence of solutions to the \swipDP
scheme in Eqs.~\eqref{eq:ch1ndtDP}--\eqref{eq:ch2ndtDP}.
\begin{theorem}[Existence of solutions to \swipDP]\label{thm:existence}
Let $\pf_h^{0} \in V_h^\polyord$ be a given initial phase-field. Then there exists
a solution $(\pf_h^1, \chem_h^1) \in V_h^\polyord \times V_h^\polyord$ to the \swipDP
scheme in Eqs.~\eqref{eq:ch1ndtDP}--\eqref{eq:ch2ndtDP}.
\end{theorem}
\begin{proof}
It is sufficient to establish existence for $n = 0$.
Following an approach analogous to~\cite[Theorem 4.1]{Barett:1999}, \cite[Theorem
2, Lemma 3]{Elliott:2000}, and \cite[Theorem 1]{dai:2016}, we first solve a regularized
problem with strictly positive mobility and subsequently pass to the degenerate limit.
For $\delta \in (0,1)$, define $\mob_\delta(\pf) = \max\{\mob(\pf), \delta\}$ and
consider the regularized scheme: find $(\pf_h^{1, \delta}, \chem_h^{1, \delta}) \in
V_h^\polyord \times V_h^\polyord$ such that
\begin{eqnarray}
  \frac{1}{\dt}\innerProd{\pf_h^{1, \delta} - \pf_h^{0}}{\vp} + \Peclet^{-1} b(\mob_\delta(\Pi_\grid^0
\pf_h^{1, \delta}), \chem_h^{1, \delta}, \vp) &=& 0, \label{eq:regch1} \\
  \innerProd{\chem_h^{1, \delta}}{\xi} - \innerProd{W_c'(\pf_h^{1, \delta}) + W_e'(\pf_h^{0})}{\xi}
- \Cahn^2 a(\pf_h^{1, \delta}, \xi) &=& 0, \label{eq:regch2}
\end{eqnarray}
for all $\vp, \xi \in V_h^\polyord$. Since $\mob_\delta(\proj_\grid^0 \pf_h^{1, \delta})
\geq \delta > 0$, Theorem~\ref{thm:coercivity} applies with $\tilde{\grid}(\pf_h^{1, \delta}) = \grid$.
Let $N_h = \dim(V_h^\polyord)$ and $\{\phi_j\}_{j=1}^{N_h}$ be the piecewise orthogonal
basis of the broken polynomial space of $V_h^\polyord$.
Testing Eq.~\eqref{eq:regch2} with $\xi = \phi_j$ and using orthogonality gives
\begin{equation}\label{eq:mumap}
  c_j = |\elem^{(j)}|^{-1}\bigl(\innerProd{W_c'(\pf_h^{1,\delta}) + W_e'(\pf_h^{0})}{\phi_j}
+ \Cahn^2 a(\pf_h^{1,\delta}, \phi_j)\bigr),
\end{equation}
where $\elem^{(j)}$ is the element supporting $\phi_j \in \Po^p(\elem^{(j)})$. We note that, by construction,
$\chem_h^{1,\delta} = \sum_{j = 1}^{N_h} c_j \phi_j =: \mathcal{S}[\pf_h^{1,\delta}:
\pf_h^0]$ is determined by $\pf_h^{1,\delta}$ given $\pf_h^0$, and $\mathcal{S}:
V_h^\polyord \to V_h^\polyord$ is continuous. Substituting into Eq.~\eqref{eq:regch1}
reduces the system to a single equation in $\pf_h^{1,\delta}$. Moreover, identifying
$V_h^\polyord \cong \R^{N_h}$ via the basis $\{\phi_j\}_{j = 1}^{N_h}$, we write
$\pf_h^{1,\delta} = \sum_{j=1}^{N_h} d_j \phi_j$ for a coefficient vector $\mb{d}
= (d_1, \ldots, d_{N_h})^T \in \R^{N_h}$. To prove existence, we apply the Leray--Schauder
fixed-point theorem (see \cite[Theorem 11.3]{Gilbarg:2001} and \cite[Theorem 2.6]{Acosta:2021})
following similarly to the proof in \cite[Theorem 2.5]{Acosta:2021}. For $\lambda
\in [0,1]$, we reformulate Eqs.\eqref{eq:regch1}--\eqref{eq:mumap} to finding $\mb{d}^{(\lambda)}
\in \R^{N_h}$ such that
\begin{equation}\label{eq:parameterized}
  d_j^{(\lambda)} = d_j^{0} - \frac{\lambda \dt \Peclet^{-1} b\bigl(\mob_\delta(\Pi_\grid^0
|\mathcal{S}[\pf_h^{(\lambda)}: \pf_h^0], \phi_j\bigr)}{|\elem^{(j)}|},
\quad j = 1, \ldots, N_h,
\end{equation}
where, implicitly, $\pf_h^{(\lambda)} = \sum_i d_i^{(\lambda)} \phi_i$ defines a
fixed-point formulation for the map $T_\lambda: \R^{N_h} \to \R^{N_h}$ satisfying
$T_\lambda(\mb{d}^{(\lambda)}) = \mb{d}^{(\lambda)}$. We now verify the hypotheses
of the Leray--Schauder fixed-point theorem. First, $T_\lambda$ is continuous for
each $\lambda \in [0,1]$, since it is a composition of continuous functions of $d^{(\lambda)}_j$.
Moreover, the map $(\lambda, \mb{d}) \mapsto T_\lambda(\mb{d})$ is jointly continuous,
as the right-hand side of Eq.~\eqref{eq:parameterized} depends continuously on both
$\lambda$ and $\mb{d}^{(\lambda)}$. Since $V_h^\polyord$ is finite-dimensional, every continuous map
on a bounded subset of $\R^{N_h}$ has a compact image; hence $T_\lambda$ is compact
for $\lambda \in [0,1]$. Second, at $\lambda = 0$ the map $T_\lambda$ simplifies
to $T_0(\mb{d}) = \mb{d}^0$ for all $\mb{d} \in \R^{N_h}$, so that $\mb{d}^0$ is
trivially a fixed point of $T_0$. Third, we show that the set of fixed points $\{\mb{d}^{(\lambda)}:
\lambda \in [0,1]\}$ is uniformly bounded in $\R^{N_h}$. To establish this,
we note that the energy identity underlying Corollary~\ref{cor:energyEstimate} remains
valid for the regularized mobility $\mob_\delta$ with the rescaling $\Peclet^{-1}
\to \lambda \Peclet^{-1}$, since $0 < \delta \leq \mob_\delta \leq 1$ uniformly in
$\delta$. For $\lambda = 0$ the bound is immediate as $\pf_h^{(0)} = \pf_h^0$. Next we establish an a priori
bound for the coefficient vector $\mb{d}^{(\lambda)}$. Lemmas~\ref{lem:orthogonality}
and \ref{lem:uniformBound} then give the inequality chain
\begin{equation}\label{eq:banach}
|\mb{d}^{(\lambda)}|^2 \min_{\elem \in \grid} |\elem| \leq \norm{\pf_h^{(\lambda)}}_{L^2(\grid)}^2 \leq C_\pf[\pf_h^0,
\Cahn, |\grid|],
\end{equation}
due to orthogonality of the basis $\{\phi_j\}_{j = 1}^{N_h}$, uniformly in $\lambda \in [0,1]$ and $\delta \in (0,1)$. Setting $R = \sqrt{\frac{C_\pf}{\min_{\elem \in \grid} |\elem|}}
+ 1$, the solution set $\{\mb{d}^{(\lambda)}: \lambda \in [0,1]\}$ is contained in the closed ball $\overline{B_R} \subset \R^N$ induced by the Euclidian norm in Eq.\eqref{eq:banach}.
By the Leray--Schauder theorem formulation in \cite[Theorem 11.3]{Gilbarg:2001},
there exists a fixed point, and thus, a solution $(\pf_h^{1,\delta}, \chem_h^{1,\delta})$
to the regularized problem in Eqs.~\eqref{eq:regch1}--\eqref{eq:regch2}.\par

We now establish existence for the degenerate case by passing to the limit $\delta
\to 0^+$. Consider a sequence $(\delta_i)_{i = 0}^\infty$ with $\delta_0 = \delta$ from the regularized problem,
$\delta_i \in (0,1)$ for $i \in \N$, and limit $\delta_i \to 0^+$. Since the uniform bound
$\norm{\pf_h^{1,\delta_i}}_{L^2(\grid)}^2 \leq C_\pf$ and the compact embedding of $\overline{B_R}
\subset \R^N$ established above are both $\delta$-independent, and the image $\mathcal{S}(\overline{B_R}:
\pf_h^{0})$ is compact in $\R^N$ by continuity of $\mathcal{S}$, the existence of
a sequence $(\delta_i)_{i = 0}^\infty$ is established such that $(\pf_h^{1,\delta_i},
\chem_h^{1,\delta_i}) \to (\pf_h^1, \chem_h^1)$ in $V_h^\polyord \times V_h^\polyord$.
It remains to verify that all terms in Eqs.\eqref{eq:regch1}--\eqref{eq:regch2} pass
to the limit. Since $\Pi_\grid^0$ is a continuous linear projection, $\Pi_\grid^0
\pf_h^{1,\delta_i} \to \Pi_\grid^0 \pf_h^1$ in $V_h^0$, and $\mob_{\delta_i}(s) =
\max\{1 - s^2, \delta_i\} \to \mob(s)$ pointwise for each $s \in \R$. For the harmonic
averages, the bound $\min\{\mob_{\delta_i}^+, \mob_{\delta_i}^-\} \leq \haver{\mob_{\delta_i}}
\leq 2\min\{\mob_{\delta_i}^+, \mob_{\delta_i}^-\}$ following Eq.~\eqref{eq:harmonicAverageEstimates1} on each $\isec \in \Gamma_i$
gives $\haver{\mob_{\delta_i}} \to \haver{\mob}$ which also covers division by zero by comparison. In particular, if $\min\{\mob^+,
\mob^-\} > 0$ this follows from continuity of $\haver{\cdot}$, while for $\min\{\mob^+,
\mob^-\} = 0$ the limit gives $\delta_i \leq \haver{\mob_{\delta_i}} \leq 2\delta_i$
which passes to the limit following Eq.~\eqref{eq:harmonicAverageEstimates1} as $\delta_i \to 0^+$. Since all integrals
in $b$ are uniformly bounded, the dominated convergence theorem \cite[E.3 Theorem
5]{Evans:2010} yields convergence of $b(\cdot,\cdot,\cdot)$. The remaining terms
are continuous or polynomial in $\pf_h^{1,\delta_i}$ and thus pass to the limit.
Therefore $(\pf_h^1, \chem_h^1)$ solves Eqs.~\eqref{eq:ch1ndtDP}--\eqref{eq:ch2ndtDP}.
\end{proof}
\begin{remark}[Uniqueness]\label{rem:unique}
Uniqueness does not follow directly from Theorem \ref{thm:existence}; however, it was
proven to be $\dt$-independent for a similar formulation (see for instance the discrete
setting studied in \cite{Acosta:2021,Barett:1999}), and the Eyre splitting using the convex-concave decomposition from Eq.~\eqref{eq:convexConcave} with the prescribed evaluations of $W_c'$ and $W_e'$
gives unique solutions~\cite{Eyre:1997}. We therefore rely on these results to allude
to uniqueness of the \swipDP scheme without providing a formal proof. Regardless,
in Section \ref{sec:experiments} we will still restrict ourselves to conservative
choices of $\dt$.
\end{remark}
Following the establishment of existence in Theorem \ref{thm:existence} and our discussion about uniqueness in Remark \ref{rem:unique} we proceed with the main results. In Section \ref{sec:dmp}, we describe how to obtain boundedness in $L_1^\infty$ for
solutions $\pf_h^1$ by applying a scaling limiter, while the cell-averaged bound
$|\proj_\grid^0 \pf_h^1| \leq 1$ a.e.\ in $\grid$ follows automatically from $|\proj_\grid^0
\pf_h^0| \leq 1$ a.e.\ in $\grid$. The existence result in Theorem \ref{thm:existence} for the
\swipDP scheme enables the application of this theory.


\section{Discrete maximum principle}\label{sec:dmp}
To obtain boundedness in $L_1^\infty$  for $\pf_h \in V_h^\polyord$ which are solutions to the \swipDP scheme, it suffices
to ensure boundedness in $L_1^\infty$ for the cell averages
$\Pi_\grid^0 \pf_h$ and subsequently apply a scaling limiter introduced in Section
\ref{sec:scalinglimiter} to enforce the point-wise bound.
However, this approach comes at the cost that a monotone energy dissipation
law can no longer be formally guaranteed for the limited scheme as we discuss in Remark \ref{rem:limitedEnergyDissipation},
although a discrete energy estimate remains available for $\polyord \geq 1$, wherein
the \swipDPL scheme is needed for a DMP. However, for $\polyord = 0$ the \swipDP
scheme gives provable DMP, mass conservation, and energy dissipation which means that it has provable
structure preservation.
\subsection{Scaling limiter}\label{sec:scalinglimiter}
The Zhang--Shu scaling limiter was first introduced in \cite{Zhang:2010} and successfully
applied
in \cite{Cheng:2013, twophase:18, Gunnarsson:2026, Liu:2024}.
The underlying idea is to rescale the phase-field on each element so that prescribed
minimum and maximum values are satisfied.
For $K \in \grid$, we define the projection operator
$\Pi_{s}: V_h^\polyord \to V_h^\polyord$ by
\begin{equation}
  \label{eq:scalinglimiter}
  \sum_{K \in \grid} \int_{K} \Pi_s[\pf_K] \cdot \varphi  =  \sum_{K \in \grid}
\int_{K} \limpff \cdot \varphi
  \quad \forall \varphi \in V_h^\polyord
\end{equation}
with the scaled phase-field
\begin{equation}
  \label{eq:stabsol}
  \limpff(\ics) :=  \pfmean + \stabfactor_K \tilpf_{\elem},
\end{equation}
and the scaling factor $\stabfactor_K$ is defined as
\begin{equation}\label{eq:limiter}
  \stabfactor_K := \min_{\ics \in \hat{\elem}}\left \{ 1, \left |\frac{\pfmean
- \pf_{\min}}{\pfmean - \pf_K(\ics)} \right |, \left |\frac{\pfmean - \pf_{\max}}{\pfmean
- \pf_K(\ics)}\right | \right \}
\end{equation}
where $\hat{\elem}$ is a sampling set, in principle taken as $\overline{K} := \elem
\cup \partial \elem$,
but in practice is typically replaced by a finite set of evaluation points for computational
efficiency -- for instance, one may take $\hat{K}$ as the set of quadrature
points $\Lambda_K \subset \overline{\elem}$ or $\Lambda_\elem = \partial \elem$ for $p = 1$ such as in \cite{Gunnarsson:2026}. Since we seek a phase-field $\pf_h \in L_1^\infty$,
we set $\pf_{\min} = -1$ and $\pf_{\max} = 1$.
\begin{lemma}[Scaling limiter bounds] \label{lem:scalinglimiterbounds}
Suppose that $\pfmean \in L_1^\infty(\elem)$ for each $\elem \in \grid$ and that
$\pf_K$ is not necessarily bounded in $L_1^\infty(\elem)$ for some $\elem \in \grid$. Then $\limpff \in L_1^\infty(\elem)$.
\end{lemma}
\begin{proof}
Since the scaling limiter acts locally, it is enough to consider a single $\elem
\in \grid$ and assume that $\pfmean \in L_1^\infty(\elem)$ while $\pf_K$ is unbounded
in $L_1^\infty(\elem)$. The $|\pfmean| = 1$ case is trivial, suppose that $|\pfmean|
< 1$ and that $\pf_\elem$ is unbounded in $L_1^\infty(\elem)$. Then, by Eqs.~\eqref{eq:stabsol}
and \eqref{eq:limiter},
\begin{equation}
-\stabfactor_\elem |\tilpf_\elem(\ics)| \leq \limpff(\ics) - \pfmean  \leq \stabfactor_\elem
|\tilpf_\elem(\ics)|
\end{equation}
and thus $|\limpff(\ics)| \leq 1$ for almost all $\ics \in \hat{\elem}$. We refer
to the original paper \cite{Zhang:2010} for details.
\end{proof}
\begin{remark}[Bound equivalences]\label{rem:bounds}
As a direct consequence of Lemma \ref{lem:scalinglimiterbounds}, if $\Pi_\grid^0\pf_h$
is bounded in $L_1^\infty(\grid)$,
then the limited phase-field $\widehat{\pf_h}$ also satisfies $\norm{\widehat{\pf_h}}_{L^\infty(\grid)}
\leq 1$. Hence, to enforce the bound on the cell averages via the \swipDP scheme,
i.e., to verify that $\norm{\Pi_\grid^0 \pf_h}_{L^\infty(\grid)} \leq 1$ prior to applying
the limiter, then $|\widehat{\pf_h}| \leq 1$ a.e.\ in $\grid$.
\end{remark}
In view of Remark~\ref{rem:bounds}, it remains to verify that the limited scheme
\swipDPL, which applies the Zhang--Shu scaling limiter after each successful time
step of the \swipDP scheme, admits a valid initialization for the \swipDP scheme
in Eqs.\eqref{eq:ch1ndtDP}--\eqref{eq:ch2ndtDP} after limiting. Since the proof of
Theorem~\ref{thm:existence} relies on the bound in Lemma~\ref{lem:uniformBound},
it suffices to show that the energy remains bounded after applying the scaling limiter
to a phase-field $\pf_h$ whose energy $\mathcal{E}[\pf_h]$ is bounded.
\par
\begin{lemma}[Boundedness of limited energy]\label{lem:contraction}
Let $\pf_h \in V_h^\polyord$ and suppose that $\mathcal{E}[\pf^n_h] \in \R^+$. Then
$\mathcal{E}[\limpf] \in \R^+$.
\end{lemma}
\begin{proof}
We consider only the case $\polyord > 0$ since for $\polyord = 0$ we do not perform limiting. For $\polyord > 0$, first, we note that
$\norm{\widehat{\pf_h}}_{L^2(\elem)} \leq \norm{\pf_h}_{L^2(\elem)}$ for each
$\elem \in \grid$ by construction of the scaling limiter, which implies that
$\widehat{\pf_h} \in L^2(\grid)$ if $\pf_h \in L^2(\grid)$. By an argument
analogous to that in the proof of Lemma~\ref{lem:uniformBound}, and owing to the
equivalence of norms in finite-dimensional spaces, there exists a constant
$C[\grid, h] > 0$ such that
\begin{equation}
  \int_\grid W(\widehat{\pf_h}) d\ics \leq \frac{C[\grid,h]\norm{\widehat{\pf_h}}^4_{L^2(\grid)}}{4}
+ \frac{|\grid|}{4} \leq \frac{C[\grid,h] \norm{\pf_h}^4_{L^2(\grid)}}{4}
+ \frac{|\grid|}{4},
\end{equation}
for some constant $C[\grid, h] > 0$ depending on the grid $\grid$ and mesh size
$h$, and is bounded since $\pf_h \in L^2(\grid)$ following
Eq.~\eqref{eq:uniformBoundIntermediate} if $\mathcal{E}[\pf_h] \in \R^+$. For
$a$, we note that $\norm{\nabla \widehat{\pf_h}}^2_{L^2 (\elem)} =
\alpha_\elem^2 \norm{\nabla \pf_h}_{L^2(\elem)} ^2 \leq
\frac{\alpha_\elem^2}{C^\Delta \Cahn} \mathcal{E}[\pf_h]$ from the third
inequality in Lem.~\ref{cor:energyEstimate}, while the products $\aver{\nabla
\widehat{\pf_h} \cdot \mathbf{n}^+}\vjump{\widehat{\pf_h}}$ and
$\vjump{\widehat{\pf_h}}^2$ are bounded in terms of
$\norm{\pf_h}_{L^2(\grid)}^2$ and $\norm{\nabla \pf_h}_{L^2(\grid)}^2$, such
that there exists a constant $C_{\mathrm{inv}} > 0$ depending on $\max_{\isec
\in \Gamma_i} h_\isec^{-2}$, $p$, $N_\partial$, $C^\Delta$ and $\Cahn$ such
that:
\begin{equation}
  a(\widehat{\pf_h}, \widehat{\pf_h}) \leq C_{\mathrm{inv}} (\norm{\pf_h}^2_{L^2(\grid)} + \norm{\nabla \pf_h}^2_{L^2(\grid)}).
\end{equation}
Hence, each term in the energy $\mathcal{E}[\widehat{\pf_h}]$ is bounded, and consequently
$\mathcal{E}[\widehat{\pf_h}] \in \R^+$.
\end{proof}
\begin{lemma}[Limited mass conservation]\label{lem:massConservationL}
  Suppose that $\pf_h \in V_h^\polyord$ and let $\widehat{\pf_h}$ be the limited version of $\pf_h$.
Then $\int_\grid \widehat{\pf_h} d\ics = \int_\grid \pf_h d\ics$.
\end{lemma}
\begin{proof}
Since $\stabfactor_K$ is constant on each element $K \in \grid$ by construction in
Eq.~\eqref{eq:limiter}:
\begin{equation}
  \int_\grid \widehat{\pf_h} d\ics = \sum_{\elem \in \grid} \int_\elem \pfmean +
\stabfactor_K \tilpf_\elem d\ics = \sum_{\elem \in \grid} |\elem| \pfmean = \int_\grid
\pf_h d\ics,
\end{equation}
where each equality follows from Eq.\eqref{eq:stabsol} and Lemma \ref{lem:orthogonality}.
\end{proof}
\subsection{Cell-averaged bounds} For $n \geq 0$, the proof of
Theorem~\ref{thm:existence} only requires that the energy at the previous step satisfies
$\mathcal{E}[\pf_h^n] \in \R^+$. Therefore, the same argument shows existence of
solutions to the \swipDPL scheme, in which the previous solution $\pf_h^n$ is replaced
by the limited version $\widehat{\pf_h^n}$. Since $\mathcal{E}[\widehat{\pf_h^n}]
\in \R^+$ whenever $\mathcal{E}[\pf_h^n] \in \R^+$ and $m_{\widehat{\pf_h^n}} = m_{\pf_h^n}$
by Lemmas~\ref{lem:contraction} and \ref{lem:massConservationL}, respectively, mass
is preserved under limiting and the energy remains bounded.
In Theorem \ref{thm:boundedVhk} we show that the cell averages $\proj_\grid^0 \pf_h^n$
for $n \geq 0$ remain bounded in $L_1^\infty$ for the \swipDP scheme if $|\proj_\grid^0
\pf_h^0| \leq 1$ a.e., which in turn implies that the limited scheme \swipDPL yields
boundedness in $L_1^\infty$ for limited solutions $\widehat{\pf_h^n} \in V_h^\polyord$.
\begin{theorem}[Boundedness in $V_h^\polyord$ for \swipDPL]\label{thm:boundedVhk}
Suppose that the initial phase-field $\pf_h^0 \in V_h^\polyord$ satisfies
 $\norm{\proj_\grid^0 \pf_h^0}_{L^\infty(\grid)} \leq 1$. Then the cell-averaged solution
$\Pi_\grid^0 \pf_h^n$ of the \swipDP scheme in Eqs.~\eqref{eq:ch1ndtDP}--\eqref{eq:ch2ndtDP}
satisfies $\norm{\Pi_\grid^0
\pf_h^{n+1}}_{L^\infty(\grid)} \leq 1$ for all $n \geq 0$, unconditionally. Consequently,
the limited version $\widehat{\pf_h^{n+1}}$ of the solution $\pf_h^{n+1}$ to the \swipDPL scheme
satisfies $\norm{\widehat{\pf_h^{n+1}}}_{L^\infty(\grid)} \leq 1$ for all $n \geq 0$.
\end{theorem}
\begin{proof}
We proceed by contradiction and with the assumption that
$\norm{\Pi_\grid^0 \pf_h^0}_{L^\infty(\grid)} \leq 1$. Suppose, for the sake
of contradiction, that there exists an element $\elem^\star \in \grid$ such that
\begin{equation}
  \pfmeanstar[1] = 1 + \varepsilon = \max_{\elem \in \grid} \pfmean[1],
\end{equation}
for some $\varepsilon > 0$ given a solution $\pf_h^1 \in V_h^\polyord$ to the \swipDP scheme in Eqs.\eqref{eq:ch1ndtDP}--\eqref{eq:ch2ndtDP} for $n = 0$. The rest follows by induction. We test Eq.~\eqref{eq:ch1ndtDP} with
the test function
\begin{equation}\label{eq:testfunc}
  \vp = (\pfmeanstar[1] - 1)_\oplus \, \chi(\elem^\star) \in V_h^0 \subset
V_h^\polyord,
\end{equation}
to obtain
\begin{equation}
  \frac{|\elem^\star|}{\dt} \innerProd{\pfmeanstar[1] - \pfmeanstar[0]}{(\pfmeanstar[1]
- 1)_\oplus} = -b(\mob(\Pi_\grid^0 \pf_h^1), \chem_h^1, (\pfmeanstar[1] -
1)_\oplus \,\chi(\elem^\star)),
\end{equation}
where, by the choice of the test function $\varphi$ in Eq.~\eqref{eq:testfunc},
\begin{eqnarray}
  &b&(\mob(\Pi_\grid^0 \pf_h^1), \chem_h^1, (\pfmeanstar[1] - 1)_\oplus \,\chi(\elem^\star))
\noindent \notag \\
  &=& \sum_{\isec \in \partial \elem^\star \cap \Gamma_i} \int_\isec \haver{\mob(\Pi_\grid^0
\pf_h^1)} (\frac{\pena}{h_\isec} \vjump{\chem_h^1} - \aver{\nabla \chem_h
\cdot \mathbf{n}^+}) (\pfmeanstar[1] - 1)_\oplus ds,
\end{eqnarray}
Since $\pfmeanstar[1] = \underset{\elem \in \grid}{\max} \pfmean[1]$ and,
consequently, $\mob(\pfmeanstar[1]) = \underset{\elem \in \grid}{\min}
\mob(\pfmean[1])$, the harmonic average satisfies following
Eq.~\eqref{eq:harmonicAverageEstimates1}:
\begin{equation}\label{eq:squeeze}
  \mob(\pfmeanstar[1])\leq \haver{\mob(\Pi_\grid^0 \pf_h^1)} \leq 2\mob(\pfmeanstar[1]).
\end{equation}
Using the comparison in Eq.\eqref{eq:squeeze} it follows from $\mob(\pfmeanstar[1])
= 0$ that
\begin{equation}
  b(\mob(\Pi_\grid^0 \pf_h^1), \chem_h^1, (\pfmeanstar[1] - 1)_\oplus \,\chi(\elem^\star))
= 0,
\end{equation}
which gives
\begin{equation}\label{eq:boundCriteria}
  \frac{|\elem^\star|}{\dt} (\pfmeanstar[1] - \pfmeanstar[0])(\pfmeanstar[1]
- 1)_\oplus = 0.
\end{equation}
Eq.\eqref{eq:boundCriteria} implies that
$\pfmeanstar[1] = \pfmeanstar[0]$ or $(\pfmeanstar[1]
- 1)_\oplus = 0$, either of which contradicts the assumption that $\pfmeanstar[1]
= 1 + \varepsilon$ for some $\varepsilon > 0$. We conclude that $\norm{\Pi_\grid^0
\pf_h^1}_{L^\infty(\grid)} \leq 1$, and by induction, $\norm{\Pi_\grid^0
\pf_h^{n+1}}_{L^\infty(\grid)} \leq 1$ for all $n \geq 0$. Finally, since $\norm{\Pi_\grid^0 \pf_h^{n+1}}_{L^\infty(\grid)}
\leq 1$, we obtain from Lemma \ref{lem:scalinglimiterbounds} that $\norm{\widehat{\pf_h^{n+1}}}_{L^\infty(\grid)}
\leq 1$ for all $n
\geq 0$. An analogous argument applies for the global minimum.
\end{proof}
\begin{remark}[Advection Cahn--Hilliard]\label{rem:convectiveBoundedness}
For the advective Cahn--Hilliard equations with a velocity field $\mbu \in H^1(\text{div},
\grid)$ with $\nabla \cdot \mbu = 0$ and $\mbu \cdot \mathbf{n}|_{\partial \grid} = 0$, then adding $c: H^1(\text{div}, \grid) \times V_h^0 \times V_h^p \to \R$:
\begin{eqnarray}
  c(\mbu, \proj_\grid^0 \pf_h, \varphi) &=& \int_{\grid} \mbu \cdot \nabla \varphi
\proj^0_\grid \pf_h d\ics \\
  &-& \sum_{\isec \in \Gamma_i} \int_\isec \left( \aver{\mbu \cdot \mathbf{n}^+}_{\oplus}
(\proj_\grid^0 \pf_h)^{+}
+ \aver{\mbu \cdot \mathbf{n}^+}_{\ominus} (\proj_\grid^0 \pf_h)^-    \right)\vjump{\varphi} ds,
\end{eqnarray}
to Eq.~\eqref{eq:ch1ndtDP}, the same DMP result holds for the \swipDPL
scheme owing to upwinding. The proof steps follow by a contradiction as used in Thm.~\ref{thm:boundedVhk} under the choice of test functions, namely that for $\pfmeanstar[1] = \underset{\elem \in \grid}{\max} \pfmean[1] = 1 + \epsilon$ for some $\epsilon > 0$ then:
\begin{equation}
  c(\mbu, \proj_\grid^0 \pf_h, (\pfmeanstar[1] - 1)_\oplus \chi_{\elem^\star}) = 0,
\end{equation}
and likewise holds in case $\pfmeanstar[1] = \underset{\elem \in \grid}{\min} \pfmean[1]$ with $\varphi = (\pfmeanstar[1] + 1)_\ominus \chi_{\elem^\star}$. The proof steps follow similarly to~\cite[Proposition 2.8]{Acosta:2021}. We leave numerical experiments to future work.
\end{remark}
\begin{remark}[Limited energy dissipation]\label{rem:limitedEnergyDissipation}
While the unlimited \swipDP scheme satisfies monotone energy dissipation -- that is,
 $\mathcal{E}[\pf_h^{n+1}] \leq \mathcal{E}[\pf_h^n] \leq  \ldots \leq \mathcal{E}[\pf_h^0]$
 for a sequence of phase-fields $(\pf_h^n)_{n \geq 0}$
 which are solutions to the \swipD scheme in Eqs.\eqref{eq:ch1ndtDP}--\eqref{eq:ch2ndtDP} for $n > 0$
 by a simple induction argument based on Theorem~\ref{thm:energyDPL} -- the
same cannot be established directly for the limited scheme
 \swipDPL (or any other limited scheme) due to the interaction of the limiter coefficients
$\stabfactor_\elem$ with the
 penalty and consistency terms
  and the data-dependent quantities in $W$. A practical trade-off thus arises: for
$\pf_h \in V_h^0$, both a DMP and energy dissipation can be
 proven,
and with
 $\polyord > 0$, boundedness is guaranteed under the limiter but energy dissipation
is not. One possible
remedy is to study the optimization problem associated with the scaling limiter,
namely to find $\tilde{\alpha}_\elem \in [0, \stabfactor_\elem]$ satisfying
\begin{equation}
  \mathcal{E}[\sum_{\elem \in \grid} \pfmean[n+1] + \tilde{\alpha}_\elem \tilpf^{n+1}_\elem]
\leq \mathcal{E}[\widehat{\pf_h^n}], \qquad n \geq 0,
\end{equation}
since these are the steps at which the energy is evaluated.
Nevertheless, it remains unclear how to enforce such a constraint, as it constitutes
a global condition on a non-convex problem. Numerical experiments, however, suggest
that this is generally not an issue, and in particular, that energy dissipation is
still observed in practice as is found and discussed in Section \ref{sec:experiments}.
\end{remark}
\subsection{Comparison to other schemes}
Drawing on the numerical observations reported in \cite{Gunnarsson:20262,Gunnarsson:2026},
we summarize below the advantages of the proposed \swipDPL scheme:
\begin{enumerate}
  \item Both the proposed \swipDPL scheme and the schemes in~\cite{Liu:2024} admit
a provable DMP for arbitrary polynomial order $p$. However,
the scheme in \cite[Section 1.5]{Liu:2024} relies on a two-step limiter, which consists
of solving a DG scheme, then solving a constrained convex problem, and finally applying
a Zhang--Shu scaling limiter. In contrast, the \swipDPL scheme requires only a single
application of the Zhang--Shu scaling limiter after each successful time step, which
is a more efficient procedure. Moreover, it is unclear whether the scheme in~\cite{Liu:2024}
satisfies an energy law, see Remark~\ref{rem:limitedEnergyDissipation}
for more details, or to what extent it remains competitive in that regard.
  \item As noted above, the proposed scheme shares the property with~\cite{Acosta:2021}
of providing boundedness for the cell-averaged solution $\Pi_\grid^0 \pf_h$ of the
\swipDP scheme, which in turn yields a DMP for the limited
scheme \swipDPL. However, the scheme in \cite{Acosta:2021} is defined only for $\pf_h
\in V_h^0$, whereas \swipDPL accommodates $\pf_h \in V_h^\polyord$ with arbitrary
polynomial order $\polyord$ while retaining the bound-preserving property. The formulation
in~\cite{Acosta:2021, Acosta:2025} (and related applications in~\cite{Acosta:2025})
is notably elegant in that one can directly establish a DMP
and energy laws even for a coupled scheme. We defer to future work the proof of analogous
results for the \swipDPL scheme, whose analysis is more involved due to the necessity
of employing a limiter for boundedness in Theorem~\ref{thm:boundedVhk} for $p > 0$
and the uncertainty regarding energy dissipation discussed in Remark \ref{rem:limitedEnergyDissipation}.
  \item For the limited schemes in~\cite{Gunnarsson:2026}, it is
unclear whether the procedure satisfies a provable DMP or an energy
law; only numerical evidence is provided in~\cite[Tab. 7]{Gunnarsson:2026}. It is also unclear how to properly treat the coercivity constant, as can be seen by comparing Theorem~\ref{thm:coercivity}
to~\cite[Theorem 3.2]{Gunnarsson:2026}. A similar remark holds for the \swipDL scheme
in~\cite[Theorem 2]{Gunnarsson:20262} for $p>0$. In contrast, the coercivity result in Theorem
\ref{thm:coercivity} does not require estimates of a data-dependent constant
to ensure semi-positivity of $\tilde{b}(\upsilon_h, \upsilon_h)$ for $\upsilon_h \in V_h^\polyord$.
\end{enumerate}

\section{Numerics}\label{sec:numerics}
We begin by stating a remark regarding some difficulties from
translating the theoretical results to the numerical setting, which may affect the
results in practice.
\begin{remark}[Numerical considerations and limitations]\label{rem:numericalConsiderations}
Below we outline some key numerical considerations and limitations of the proposed
scheme which may affect the results in practice:
\begin{enumerate}
  \item A brief runtime comparison
    with the \swipL scheme presented in \cite{Gunnarsson:2026} shows that the
    number of linear iterations inside the Newton solver stay roughly the same,
    but overall runtime of assembly and linear solver time increase
    significantly (currently factor $10$) which we attribute to the current
    implementation of the scheme in a non-optimal mixed setting adding an
    auxiliary variable for the cell averages.
    This issue is currently addressed in order to allow fair efficiency
    comparisons between the different schemes.

    \item Similarly, standard solvers for the nonlinear system require a finite
tolerance $\tol$ for both the Newton and linear iterations, which may also
affect the DMP of the \swipDPL scheme.
    \item In our Zhang--Shu scaling limiter implementation, we add a tolerance $\tol_{\lim}
= 5 \cdot 10^{-16}$ to avoid division by zero in Eq.~\eqref{eq:limiter}. We also artificially set $\stabfactor_\elem = 0$ if $|\pfmean|
> 1 - \tol_{\text{avg}}$, with a threshold of $\tol_{\text{avg}} = 10^{-14}$ to limit numerical overflow.
\end{enumerate}
\end{remark}
In view of Remark~\ref{rem:numericalConsiderations}, we also note that the \swipDPL
scheme does not necessarily satisfy a monotone energy dissipation law for polynomial
order $\polyord \geq 1$ due to the scaling limiter, as noted in Remark~\ref{rem:limitedEnergyDissipation}.
However, this property is observed in practice as seen in Figs. \ref{fig:dataSP}
and \ref{fig:dataMerging}.


\subsection{Software implementation}
The numerical simulations in Section \ref{sec:experiments} are performed using the
software package \dunefem \cite{dunereview:21, dune_v210:25, dunefemdg:21} at version $2.12$ with
its \code{Python} interface \cite{Dedner:2020}. We use the limiter implementation
from \cite{dunefemdg:21} in \dunefemdg, which is very similar to the experimental
routine in \cite{Gunnarsson:20262,Gunnarsson:2026}.


\subsection{Experiments}\label{sec:experiments}
\begin{example}[Trigonometric initial data]\label{ex:trig}
Consider the manufactured solution in $\Omega = [0,1]^2$ given by:
\begin{equation}\label{eq:trig}
        \pf_I(\mbx, t) = A_1\cos(4 \pi \mbx_1) \cos(4 \pi \mbx_2) e^{-\gamma t}, \qquad \pf_0(\mbx) = \pf_I(\mbx, 0),
\end{equation}
with $\pf_0$ being the initial data and we parameterize for some amplitude $A_1 \in (0,1]$, decay $\gamma \geq 0$, and $\mbx = (\mbx_1, \mbx_2) \in \Omega$. We partition
the domain into $N \times N$ quadrilateral elements and set the parameters $\Cahn
= \frac{1}{10}$, $\Peclet = \frac{3}{10}$. We pick penalty parameter $\eta = \max\{1, 8p^2 + \frac{1}{2}\}$.
\end{example}
We investigate the experimental order of convergence (EOC) both for case $\gamma = 0$ and $\gamma > 0$ to cover the spatial EOC in the time-independent and time-dependent cases.  At each refinement level,
we double the number of elements in each direction, i.e. $N \to 2N$, and on the coarsest mesh we pick $\tau = \tau_0$ and $\tau \to \frac{\tau}{2}$ when the grid refines. For the EOC, we compute the $L^2$ and $H^1$ errors at the final time $T$ with
respect to the manufactured solution $\pf_I$ in Eq.~\eqref{eq:trig}. The EOC is then computed as
of the forced CH equations, where the source term, following~\cite{Wimmer:2025}, is given by:
\begin{equation}
S_{\pf_I}(\mbx, t)= -\gamma \pf_I(\mbx, t) -\Peclet^{-1}\nabla
\cdot (\mob(\pf_I) \nabla (W'(\pf_I) - Cn^2 \Delta \pf_I)),
\end{equation}
has been added to the right-hand side of Eq.~\eqref{eq:ch1ndtDP}.
\begin{table}[htbp]
  \centering
  \caption{Ex.\ref{ex:trig}: $L^2$ and $H^1$ errors and EOC for $A_1 \in \{0.1, 0.99\}$ and $\gamma = 0$ with $p \in \{0,1,2\}$, $T = 10^{-4}$, and $\tau_0 = \frac{T}{2}$.}
  \label{tab:amp_joint}
  \footnotesize
  \begin{tabular}{l @{\hspace{0.6em}} c @{\hspace{0.5em}} c @{\hspace{0.5em}} c @{\hspace{0.5em}}
c @{\hspace{0.5em}} c @{\hspace{0.5em}} c @{\hspace{0.5em}} c @{\hspace{0.5em}} c
@{\hspace{0.5em}} c @{\hspace{0.5em}} c}
    \hline\hline
    & & \multicolumn{4}{c}{$A_1 = 0.1$} & \multicolumn{4}{c}{$A_1 = 0.99$} \\
    \cline{3-6} \cline{7-10}
    & & \multicolumn{2}{c}{$L^2$ Error} & \multicolumn{2}{c}{$H^1$ Error} & \multicolumn{2}{c}{$L^2$ Error} & \multicolumn{2}{c}{$H^1$ Error} \\
    Scheme, $p$ & $N$ & Error & EOC & Error & EOC & Error & EOC & Error & EOC \\
    \hline
                                   &    40 &   6.41e-03 &    --- &        --- &    --- &   6.34e-02 &    --- &        --- &    --- \\
                                   &    80 &   3.21e-03 &   1.00 &        --- &    --- &   3.17e-02 &   1.00 &        --- &    --- \\
    \swipDP, $0$    &   160 &   1.60e-03 &   1.00 &        --- &    --- &   1.59e-02 &   1.00 &        --- &    --- \\
                                   &   320 &   8.02e-04 &   1.00 &        --- &    --- &   7.94e-03 &   1.00 &        --- &    --- \\
    \hline
                                   &    40 &   2.62e-04 &    --- &   8.08e-02 &    --- &   2.63e-03 &    --- &   8.01e-01 &    --- \\
                                   &    80 &   6.53e-05 &   2.00 &   4.03e-02 &   1.00 &   6.48e-04 &   2.02 &   3.99e-01 &   1.00 \\
    \swipDPL, $1$    &   160 &   1.63e-05 &   2.00 &   2.01e-02 &   1.00 &   1.62e-04 &   2.00 &   1.99e-01 &   1.00 \\
                                   &   320 &   4.08e-06 &   2.00 &   1.01e-02 &   1.00 &   4.04e-05 &   2.00 &   9.97e-02 &   1.00 \\
    \hline
                                   &    40 &   1.12e-05 &    --- &   3.27e-03 &    --- &   1.34e-03 &    --- &   1.40e-01 &    --- \\
                                   &    80 &   1.41e-06 &   3.00 &   8.17e-04 &   2.00 &   3.13e-05 &   5.42 &   8.21e-03 &   4.09 \\
    \swipDPL, $2$    &   160 &   1.76e-07 &   3.00 &   2.04e-04 &   2.00 &   7.20e-06 &   2.12 &   2.05e-03 &   2.00 \\
                                   &   320 &   2.20e-08 &   3.00 &   5.11e-05 &   2.00 &   1.77e-06 &   2.03 &   5.13e-04 &   2.00 \\
    \hline\hline
  \end{tabular}
\end{table}
\begin{table}[htbp]
  \centering
  \caption{Ex.~\ref{ex:trig}: $L^2$ and $H^1$ errors and EOC for $\gamma \in \{\frac{1}{4},1\}$, and $A_1 = 0.1$ with $p \in \{0,1,2\}$, $T = 10^{-3}$, and $\tau_0 = \frac{T}{20}$.}
  \label{tab:td_gamma_joint}
  \footnotesize
  \begin{tabular}{l @{\hspace{0.6em}} c @{\hspace{0.5em}} c @{\hspace{0.5em}} c @{\hspace{0.5em}}
c @{\hspace{0.5em}} c @{\hspace{0.5em}} c @{\hspace{0.5em}} c @{\hspace{0.5em}} c
@{\hspace{0.5em}} c @{\hspace{0.5em}} c}
    \hline\hline
    & & \multicolumn{4}{c}{$\gamma = 1/4$} & \multicolumn{4}{c}{$\gamma = 1$} \\
    \cline{3-6} \cline{7-10}
    & & \multicolumn{2}{c}{$L^2$ Error} & \multicolumn{2}{c}{$H^1$ Error} & \multicolumn{2}{c}{$L^2$ Error} & \multicolumn{2}{c}{$H^1$ Error} \\
    Scheme, $p$ & $N$ & Error & EOC & Error & EOC & Error & EOC & Error & EOC \\
    \hline
                                   &    40 &   6.42e-03 &    --- &        --- &    --- &   6.41e-03 &    --- &        --- &    --- \\
                                   &    80 &   3.21e-03 &   1.00 &        --- &    --- &   3.20e-03 &   1.00 &        --- &    --- \\
    \swipDP, $0$  &   160 &   1.60e-03 &   1.00 &        --- &    --- &   1.60e-03 &   1.00 &        --- &    --- \\
                                   &   320 &   8.01e-04 &   1.00 &        --- &    --- &   8.01e-04 &   1.00 &        --- &    --- \\
    \hline
                                   &    40 &   3.20e-04 &    --- &   8.06e-02 &    --- &   3.20e-04 &    --- &   8.05e-02 &    --- \\
                                   &    80 &   8.07e-05 &   1.99 &   4.03e-02 &   1.00 &   8.06e-05 &   1.99 &   4.02e-02 &   1.00 \\
    \swipDPL, $1$   &   160 &   2.00e-05 &   2.01 &   2.01e-02 &   1.00 &   2.00e-05 &   2.01 &   2.01e-02 &   1.00 \\
                                   &   320 &   5.01e-06 &   2.00 &   1.01e-02 &   1.00 &   5.00e-06 &   2.00 &   1.01e-02 &   1.00 \\
    \hline
                                   &    40 &   1.12e-05 &    --- &   3.27e-03 &    --- &   1.12e-05 &    --- &   3.26e-03 &    --- \\
                                   &    80 &   1.41e-06 &   3.00 &   8.17e-04 &   2.00 &   1.41e-06 &   3.00 &   8.16e-04 &   2.00 \\
    \swipDPL, $2$   &   160 &   1.76e-07 &   2.99 &   2.05e-04 &   1.99 &   1.82e-07 &   2.95 &   2.04e-04 &   2.00 \\
                                   &   320 &   2.29e-08 &   2.95 &   5.52e-05 &   1.89 &   3.38e-08 &   2.43 &   5.10e-05 &   2.00 \\
    \hline\hline
  \end{tabular}
\end{table}
Table \ref{tab:amp_joint} shows the EOC for the \swipDP for $p = 0$ and \swipDPL schemes with
polynomial orders $p \in \{1,2\}$, for two different amplitudes
$A = 0.1$ and $A = 0.99$ in the time-independent case $\gamma = 0$. We observe that the \swipDP scheme with $\polyord = 0$ exhibits a first-order
convergence in the $L^2$ norm for $A \in\{0.1, 0.99\}$. Meanwhile, the \swipDPL scheme with $\polyord \in \{1,2\}$ exhibits a $(p+1)$-order
convergence in the $L^2$ norm and a $p$-order convergence in the $H^1$ norm, for $A_1 = 0.1$, and for $p = 1$, also for the $A_1 = 0.99$ case. However, asymptotically for $A = 0.99$ the $L^2$ convergence rate for the $p = 2$ case is not optimal, as observed in Table~\ref{tab:amp_joint}, but the $H^1$ convergence remains consistent. We also find similar observations in the time-dependent case in Tab.~\ref{tab:td_gamma_joint} wherein we get lower order of convergence for the $L^2$ norm for $p = 2$ with $\gamma = 1$ and $A_1 = 0.1$, but is close to optimal in the $\gamma = \frac{1}{4}$ case. For $p \in \{0,1\}$ results are optimal even in the time-dependent case, as found in Tab.~\ref{tab:td_gamma_joint}.
Regardless, the un-optimal $p$-order $L^2$ convergence for the $p = 2$ case at $A_1 = 0.99$ with $\gamma = 0$ and $A_1 = 0.1$ with $\gamma = 1$ is also similar to the \sipg findings in~\cite[Table 2]{Sariaydin:2017} which would be equivalent to picking $A_1 = 1$. We leave it for future work to investigate the cause for this observation since it appears to be time-integration related (see for instance~\cite{Liu:2021}) and we also note for the finest mesh $h^3 \approx 10^{-9}$, which causes severe $\dt$-constraints in our computations compared to the ones presented in~\cite{Liu:2021, Sariaydin:2017} which studied coarser meshes.
\begin{example}[Spinodal decomposition]\label{ex:spinodal}
We consider a triangulation in $\Omega = [0,1]^2$ and discretize it into $64 \times
64$ triangles. The initial data are given by $\pf_h^0 = 0.3 + \xi$, where $\xi \in
V_h^0$ is a random variable uniformly distributed in $[-0.01, 0.01]$ on each $\elem \in \grid$. We set the
parameters $\Cahn = \frac{1}{64}$, $\Peclet = 1$, $\dt = 10^{-4}$, and end time $T = 0.05$. We pick a global penalty parameter $\eta =  \max\{1, 3p(p+1) + \frac{1}{2}\}$.
\end{example}
\begin{figure}[h]
\centering
\includegraphics[width = 0.32\textwidth]{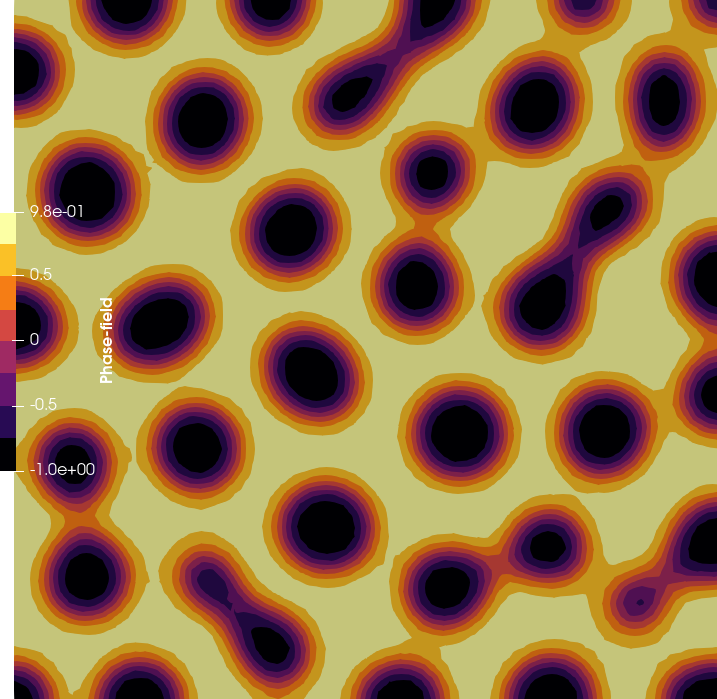}
\hfill
\includegraphics[width = 0.32\textwidth]{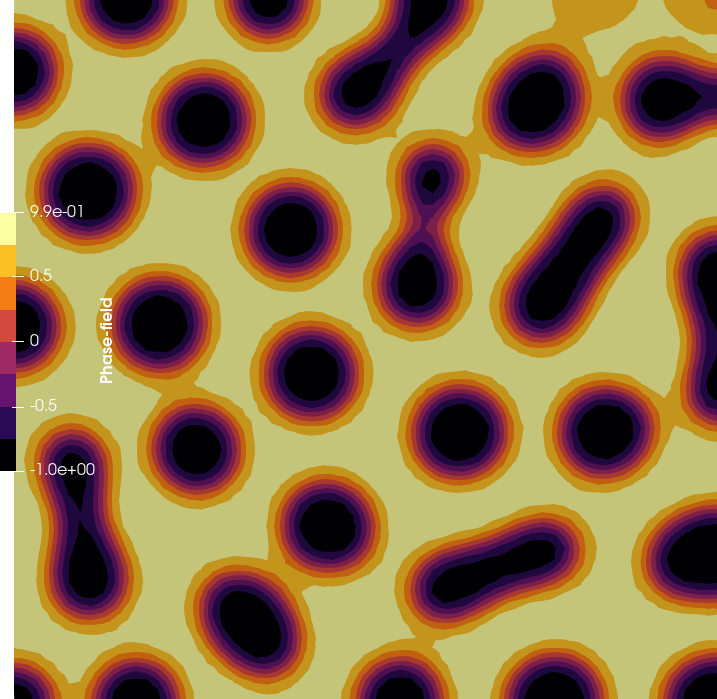}
\hfill
\includegraphics[width = 0.32\textwidth]{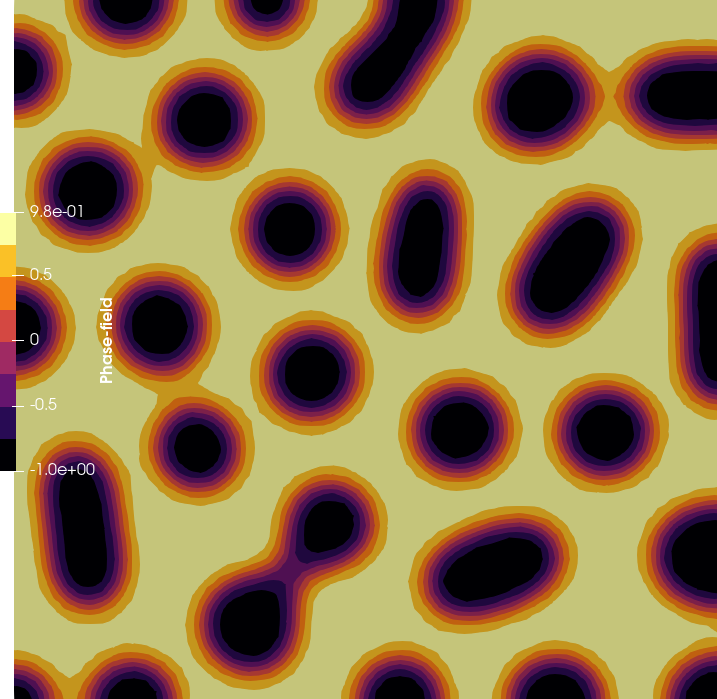}
  \caption{Example \ref{ex:spinodal}: Snapshots of the phase-field for \swipDPL ($p=1$) at time steps $t
= \frac{T}{2}$, $t = \frac{3T}{4}$, and $t = T$ (left to right).}\label{fig:rand14}
\end{figure}

\begin{figure}[h]
\centering
\includegraphics[width = 0.85\textwidth]{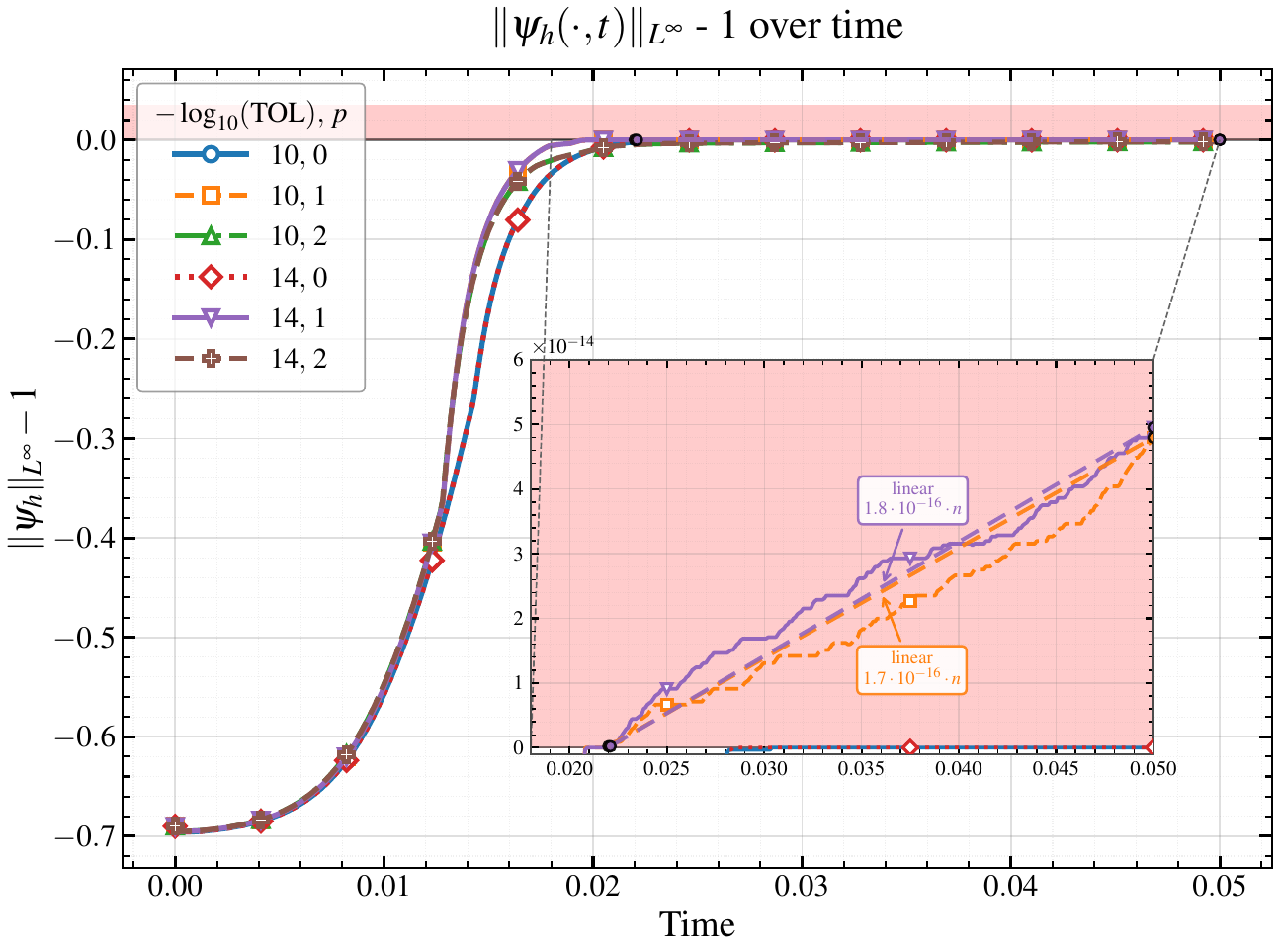} \label{fig:violation}
\caption{Example \ref{ex:spinodal}: DMP violation, we make
a linear fit to show that the error scales with computer precision in the
  floating point interval around $1$.}
\end{figure}

\begin{figure}[h]
\centering
\includegraphics[width = 0.99\textwidth]{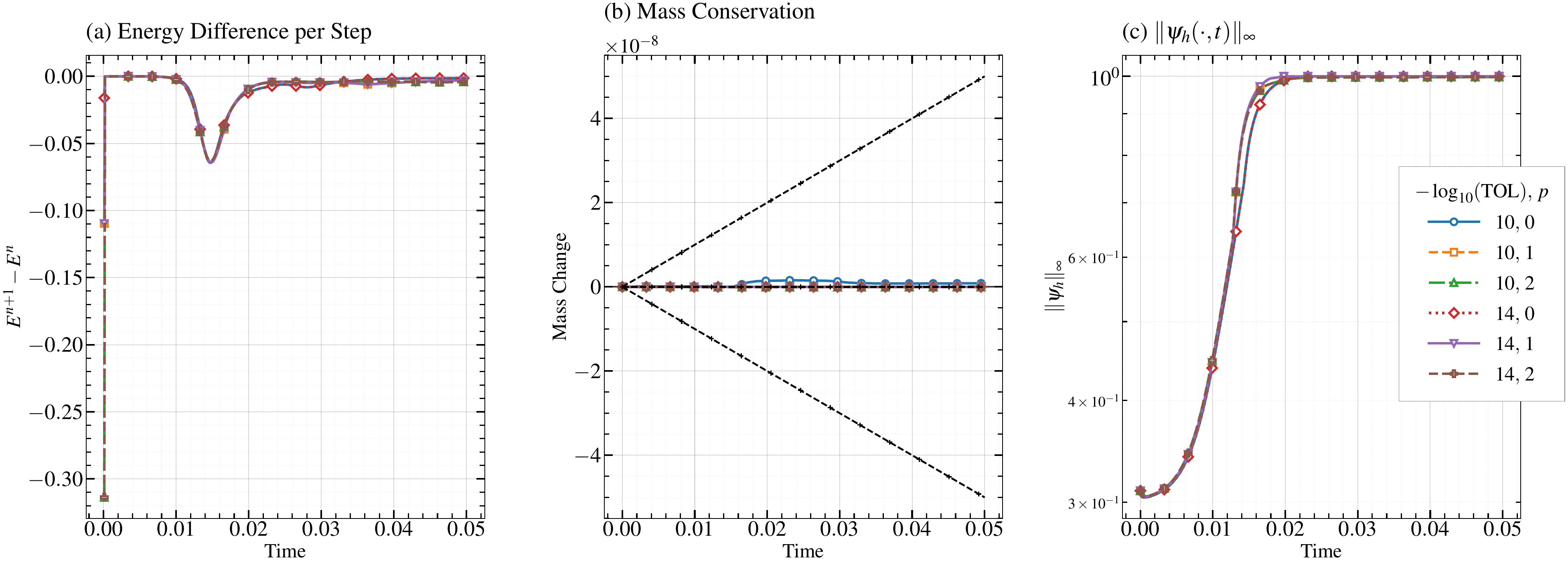} \label{fig:dataSP}
\caption{Example \ref{ex:spinodal}: Monotone energy dissipation, mass change (black line is $n\tol$), and
bounds.}
\end{figure}

For Example \ref{ex:spinodal} we observe standard spinodal decomposition in
Fig.~\ref{fig:rand14} and use this example to discuss numerical constraints to
properly fulfill Theorem~\ref{thm:boundedVhk} in practice. As can be seen in
Fig.~\ref{fig:violation} an exact bound is not fully obtained for $p = 1$ while
it is perfect for $p = 0$, and for $p = 2$ the inequality
$||\pf_h^n||_{L^\infty(\grid)} < 1 $ for each $n$. To study the DMP violation
for $p = 1$, we constructed a linear fit to show that the overall deviation and
note that violations scales as machine tolerance and therefore breaks the
structure of Lemma~\ref{lem:scalinglimiterbounds} and its application in
Theorem~\ref{thm:boundedVhk} for $p  = 1$ . We see some dependency on the
non-linear tolerance $\tol$, but we note that we cannot numerically guarantee an
exact DMP due to the presence of floating point errors and the non-linear
tolerance. A solution would be to cut off the small deviations, which would be
more appropriate than cut-off \fem (see for instance~\cite{Gunnarsson:2026}
regarding \femC) as the error is smaller than the mass deviation. However,
again, this could cause errors for the energy. On the other hand, in
Fig.~\ref{fig:dataSP}, we observe that the mass violation is most prevalent for
$p = 0$ with $\tol = 10^{-10}$, but less apparent for $p \in \{1,2\}$ at the
same tolerance. Regardless, the deviation is consistent with the discussions
in~\cite{Gunnarsson:20262,Gunnarsson:2026} regarding observed mass violations
and nonlinear tolerances. Interestingly, the energy dissipation is monotone for
$p \in \{0,1,2\}$, verifying the expected energy dissipation law despite the
missing formal result for the limiter being a contraction on the energy
$\mathcal{E}$, as discussed in Remark~\ref{rem:limitedEnergyDissipation}.
\begin{example}[Merging bubbles]\label{ex:merging}
Consider the domain $\Omega_T = [0,1]^2$ with $N = 64 \times 64$ triangles.
The initial condition is given by a smooth profile
\begin{equation}
    \pf(\mbx,0) = (1 - A_2)\left(2\min\left\{\left(1 + 2^{-1}\sum_{j=1}^{2} \tanh\left(\frac{r
- \norm{ \mathbf{x} - \mathbf{c}_j}}{\sqrt{2}Cn}\right)
\right),1\right\} - 1\right),
\end{equation}
where $r = 0.2$ is the droplet radius, with
central points $\mathbf{c}_1 = (0.3,
0.5)^T$ and $\mathbf{c}_2 = (0.7, 0.5)^T$, $A_2 \in [0,1)$ is some scaling factor, Cahn number $\Cahn = \frac{1}{64}$, and Peclet
number $\Peclet = 1$. The simulation is run with $\dt = 5 \cdot 10^{-5}$ until $T
= 0.04$. We pick a global penalty parameter $\eta = \max\{1, 3p(p+1) + \frac{1}{2}\}$.
\end{example}
\begin{figure}[h]
\centering
\includegraphics[width = 0.99\textwidth]{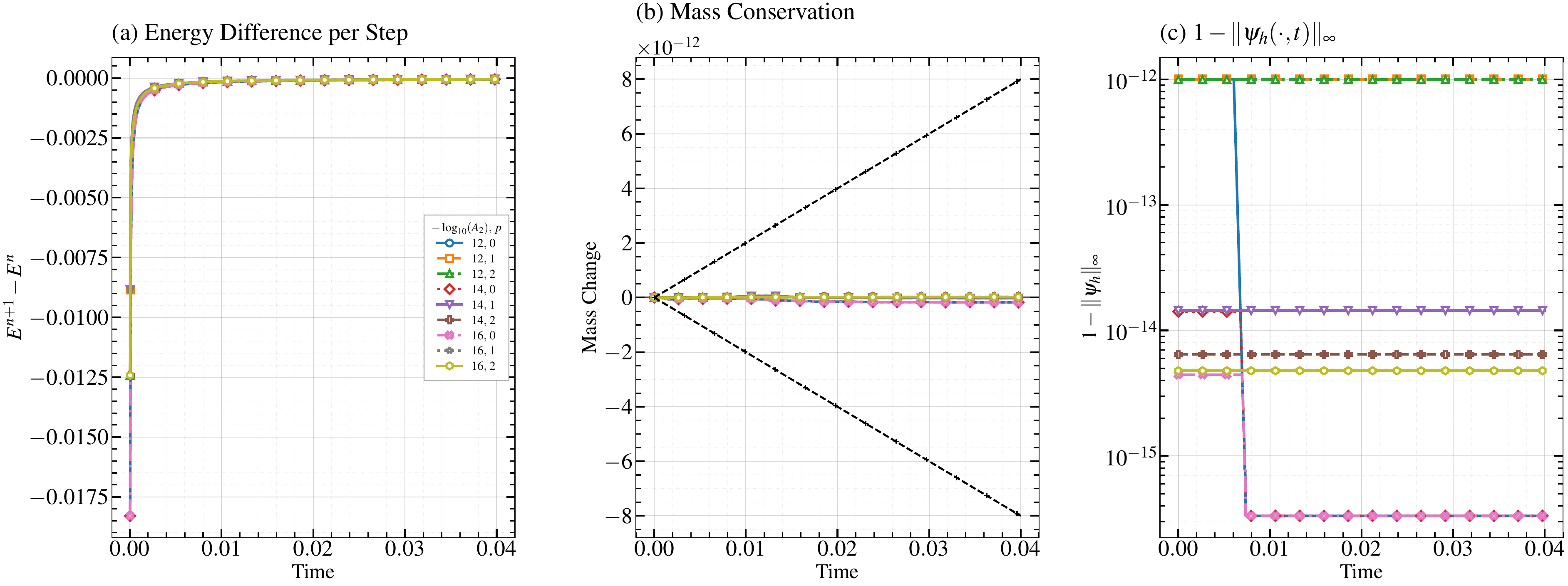}\label{fig:dataMerging}
\caption{Example \ref{ex:merging}: Monotone energy dissipation, mass change (black line is $n\tol$), and
bounds.}
\end{figure}
Fig.~\ref{fig:dataMerging} shows monotone energy dissipation, mass change, and
bounds for Example \ref{ex:merging} with $\polyord \in \{0,1,2\}$ and $\tol =
10^{-14}$. We note that $\norm{\pf(\cdot,0)}_{L^\infty(\Omega)}$ is not
necessarily equal to $\norm{\pf_h^0}_{L^\infty(\grid)}$ as the initial data
$\pf(\cdot,0)$ is $L^2$-projected onto the discrete space $V_h^\polyord$, as can
be seen in Fig.~\ref{fig:dataMerging}, which does not necessarily preserve the
initial $L^\infty$-norm. We observe results very similar to those in
\cite{Gunnarsson:20262} (where the \swipDL~ scheme was used with $A_2 =
10^{-8}$) for the $\polyord \in \{0,1,2\}$ case with similar initial data
structure, but with a modified Peclet number $\Peclet$ and $\dt$. Notably,
although the \swipDL scheme in~\cite{Gunnarsson:20262} is not provably
bound-preserving, a similar bound behaviour is observed for the \swipDPL scheme,
the bounds remain frozen at the initial time and do not oscillate for $p \in
\{1,2\}$, while for $p = 0$, we get that $||\pf_h(t,\cdot)||_{L^\infty(\grid)}$
becomes very close to $1$ but never reaches above it, similarly
to~\cite{Gunnarsson:20262}. This contrasts with the study
in~\cite{Gunnarsson:2026} (where the \swipL scheme was used with $A_2 = 10^{-2}$
due to initial value constraints), which, despite having numerical evidence of
satisfying a DMP, did not exhibit these perfectly frozen bounds
and had oscillations for the maximum and minimum values. One interpretation is
that the coercivity constraint and less diffusive structure of the \swipL scheme
derived in~\cite{Gunnarsson:2026} are not optimally formulated to preserve
bounds. This fallback is further suggested by the fact that the scheme requires
a regularization of the mobility to satisfy the coercivity constraint, and is
not provably bound-preserving.

\section{Summary and Outlook}\label{sec:summary} In this paper we constructed a
novel scheme which unconditionally satisfies a DMP for arbitrary polynomial order
$\polyord$ using a scaling limiter. Numerical examples in
Section~\ref{sec:numerics} validate the theoretical results presented in
Sections \ref{sec:discretization} and \ref{sec:scheme}. Shortcomings are
discussed in Remark \ref{rem:numericalConsiderations} and optimal order of
convergence of the proposed scheme is demonstrated in Table~\ref{tab:amp_joint}.

A key theoretical contribution is the mobility-free constant constraint on the
coercivity of the \swipDP, which removes the need for a regularization parameter
as required in~\cite{Gunnarsson:20262,Gunnarsson:2026}. Moreover, the \swipDPL
scheme reduces the complexity of the limiter process in~\cite{Liu:2024} by
eliminating the need to adjust the cell-averages, and one simply obtains the
desired constrained bounds thanks to the result in Theorem~\ref{thm:boundedVhk}.
Moreover, we also proved existence of a solution and highlighted energy
stability properties of the proposed scheme, despite only being provable to be
unconditionally monotone decreasing for the very special case when $\polyord =
0$, while for $\polyord \in\{1,2\}$ this is only supported by numerical
evidence. Optimal EOC is obtained for the $\polyord \in \{0,1\}$ case as shown
in Tab.~\ref{tab:amp_joint}. However, for the $\polyord = 2$ case, we also
observe an un-optimal EOC for the high-amplitude case in
Tab.~\ref{tab:amp_joint} which could be a limitation for the \swipDPL scheme
with $\polyord \geq 2$. However, we leave that to future work and suggest that
further investigation is needed to fully understand the behavior of the scheme
for higher-order polynomial degrees, especially in the context of high-amplitude
initial conditions and its relation to degeneracy for the \swipDPL scheme. It is
of interest to further investigate the energy stability following
Remark~\ref{rem:limitedEnergyDissipation}, and to explore the possibility of a
provable energy stability result for higher-order polynomial degrees while a DMP
is maintained. Moreover, it would be of interest to see how competitive this
scheme is against other schemes with a DMP such
as~\cite{Acosta:2021} and~\cite{Liu:2024}, and to explore the possibility of
extending the proposed scheme to the coupled Cahn--Hilliard--Navier--Stokes
system, or other coupling schemes as a potential extension aided by
Remark~\ref{rem:convectiveBoundedness} and $hp$-adaptive configurations with a
theoretically valid scheme.

\bibliographystyle{siamplain}
\bibliography{refs.bib}
\end{document}